\documentclass[12pt]{amsart}
\usepackage[sans]{dsfont}
\usepackage{amsfonts,amssymb,amsbsy,amsmath,amsthm,enumerate}

\numberwithin{equation}{section}

\newtheorem{theorem}{Theorem}[section]

\newtheorem{lemma}[theorem]{Lemma}
\newtheorem{defi}[theorem]{Definition}

\newtheorem{pr}[theorem]{Proposition}
\theoremstyle{definition}

\newcommand{\bel}{\begin{equation} \label}
\newcommand{\ee}{\end{equation}}
\newcommand{\one}{\mathds{1}}

\newcommand{\bx}{{\bf x}}

\newcommand{\rd}{{\mathbb R}^{2}}
\newcommand{\re}{{\mathbb R}}

\newcommand{\eps}{{\epsilon}}

\def\beq{\begin{equation}}
\def\eeq{\end{equation}}
\newcommand{\bea}{\begin{eqnarray}}
\newcommand{\eea}{\end{eqnarray}}
\newcommand{\beas}{\begin{eqnarray*}}
\newcommand{\eeas}{\end{eqnarray*}}

\begin{document}
\title[Eigenvalues for periodic twisting]{Spectral Asymptotics  for  Waveguides with Perturbed Periodic Twisting}

\author[G.~Raikov]{Georgi Raikov}

\begin{abstract}
We consider the twisted waveguide $\Omega_\theta$, i.e. the domain obtained by the rotation of the bounded cross section $\omega \subset \rd$ of the straight tube $\Omega : = \omega \times \re$ at angle $\theta$ which depends on the variable along the axis of $\Omega$. We study the spectral properties of the Dirichlet Laplacian in $\Omega_\theta$, unitarily equivalent under the diffeomorphism $\Omega_\theta \to \Omega$ to the operator $H_{\theta'}$, self-adjoint in ${\rm L}^2(\Omega)$.  We assume that $\theta' = \beta - \eps$ where $\beta$ is a $2\pi$-periodic function, and $\eps$ decays at infinity. Then in the spectrum $\sigma(H_\beta)$ of the unperturbed operator $H_\beta$ there is a semi-bounded gap $(-\infty, {\mathcal E}_0^+)$, and, possibly, a number of bounded gaps $({\mathcal E}_j^-, {\mathcal E}_j^+)$. Since $\eps$ decays at infinity, the essential spectra of $H_\beta$ and $H_{\beta - \eps}$ coincide. We investigate the asymptotic behaviour of the discrete spectrum of $H_{\beta - \eps}$ near an arbitrary fixed spectral edge ${\mathcal E}_j^\pm$. We establish necessary and quite close sufficient conditions which guarantee the finiteness of $\sigma_{\rm disc}(H_{\beta-\eps})$ in a neighbourhood of ${\mathcal E}_j^\pm$. In the case where the necessary conditions are violated, we obtain the main asymptotic term of the corresponding eigenvalue counting function. The effective Hamiltonian which governs the the asymptotics of $\sigma_{\rm disc}(H_{\beta-\eps})$ near ${\mathcal E}_j^\pm$ could be represented as a finite orthogonal sum of operators of the form
$$
-\mu\frac{d^2}{dx^2} - \eta \eps,
$$
self-adjoint in ${\rm L}^2(\re)$; here, $\mu > 0$ is a constant related to the so-called {\em effective mass}, while $\eta$ is $2\pi$-periodic function depending on $\beta$ and $\omega$.
\end{abstract}

\maketitle



\section{Introduction}
\label{s1} \setcounter{equation}{0}
Since the seminal work \cite{DE}, there has been an unfading interest towards the spectral properties of quantum waveguides, with an accent on the problem of existence of discrete eigenvalues. During the last decade the 3D {\em twisted waveguides} were investigated by numerous authors. Recently, a special attention has been allocated to the cases where the global twisting does not vanish, but has a non trivial asymptotic behaviour at infinity (see e.g. \cite{ExK, DeO, BP1, BKRS, BKR} and the references cited there).\\
In the present article we investigate
the asymptotic behaviour of the discrete spectrum near the edges of the essential one for
the Dirichlet Laplacian in a twisted waveguide with perturbed periodic twisting.\\
First, we describe the waveguides which we will deal with. Let $\omega \in \rd$ be a bounded domain. Introduce the straight tube  $\Omega : = \omega \times \re \subset \re^3$.
For ${\bf x} = (x_1,x_2,x_3) \in \Omega$, we write ${\bf x} = (x_t,x_3)$ with $x_t = (x_1,x_2) \in \omega$, and $x_3 \in \re$.
Assume that $\theta \in C^1(\re; \re)$, $\theta' \in {\rm L}^{\infty}(\re)$. Define the twisted tube
$$
\Omega_\theta = \{ r_\theta(x_3)\, \bx \in \re^3\, | \, \bx \in \Omega\}
$$
where
$$
 r_\theta(x_3) =
 \left(
 \begin{array}{ccc}
\cos \theta(x_3) & \sin \theta(x_3) & 0 \\
-\sin \theta(x_3) & \cos \theta(x_3) & 0 \\
0 & 0& 1
\end{array} \right).
$$
Then the Dirichlet Laplacian $-\Delta_{\Omega_\theta}^D$ is the self-adjoint operator generated in ${\rm L}^2(\Omega_\theta)$ by the closed quadratic form
$$
\tilde{\mathcal Q}_\theta[f]= \int_{\Omega_\theta} \vert \nabla f
(\bx)\vert^2 d \bx,\quad f\in {\rm D}(\tilde{\mathcal Q}_\theta)=
{\rm H}_0^1 (\Omega_\theta).
$$
Define the unitary operator ${\mathcal U}: {\rm L}^2( \Omega_\theta) \to {\rm L}^2(\Omega)$ by
$$
(\mathcal{U} f)(\bx) = f \left( r_\theta(x_3) \, \bx \right), \quad \bx \in \Omega, \quad
f \in {\rm L}^2( \Omega_\theta).
$$
Set
 $$
 \nabla_{t}:  =( \partial_{1},
\partial_{2})^T, \quad \Delta_t : = \partial_{1}^2 +
\partial_{2}^2, \quad \partial_{\varphi} := x_1
\partial_2 - x_2 \partial_1,
$$
and denote by $H_{\theta'}$ the self-adjoint operator generated
in ${\rm L}^2(\Omega)$ by the closed quadratic form
    \bel{q}
    {\mathcal Q}_{\theta'}[f] : = \tilde{\mathcal Q}_{\theta}[{\mathcal U}^{-1}f]
    = \int_{\Omega}
(\vert \nabla_{t} f \vert ^2 + \vert \theta'(x_3)
\partial_{\varphi} f  + \partial_3 f \vert^2 )\, d \bx, \quad
f \in {\rm  H}_0^1 (\Omega ).
 \ee
Then we have
$$
H_{\theta'} = {\mathcal U} \left(-\Delta_{\Omega_{\theta}}\right) {\mathcal U}^{-1}.
$$
Note that $H_{\theta'} \geq \lambda_1 I$ where $\lambda_1 > 0$ is the lowest eigenvalue of the cross-section Dirichlet Laplacian $-\Delta_t$, self-adjoint in  ${\rm L}^2(\omega$); hence,  $H_{\theta'}$ is boundedly invertible in ${\rm L}^2(\Omega)$. In \cite[Proposition 2.1]{BKR}, it was shown that if $\partial \omega \in C^2$, and $\theta \in C^2(\re)$ with $\theta', \theta'' \in {\rm L}^{\infty}(\re)$, then the domain $D(H_{\theta'})$ of $H_{\theta'}$ coincides with ${\rm H}^2(\Omega) \cap {\rm H}^1_0(\Omega)$, and
$$
H_{\theta'} = - \Delta_t - \left(\theta' \partial_\varphi + \partial_3\right)^2.
$$
In \cite{BKRS} we considered the spectral properties of $H_{\theta'}$ under the hypotheses  $\theta' = \beta - \epsilon$ where $\beta > 0$ is a constant, and $\epsilon \geq 0$ is a function which decays at infinity. Then, $H_\beta$ is unitarily equivalent under the partial Fourier transform with respect to $x_3$, to an analytically fibered operator, the spectrum $\sigma(H_\beta)$ of $H_\beta$ is purely absolutely continuous, and coincides with $[{\mathcal E}, \infty)$
(see \cite{ExK} or \cite[Subsection 2.2]{BKRS}). Since $\epsilon$ decays at infinity, the essential spectra $\sigma_{\rm ess}(H_\beta)$ and $\sigma_{\rm ess}(H_{\beta - \epsilon})$ coincide.
In \cite{BKRS} we established necessary and sufficient conditions on $\epsilon$ and the geometry of $\omega$ which guarantee the finiteness of the discrete spectrum of $H_{\beta-\epsilon}$  below ${\mathcal E}$. In the case where the necessary conditions are violated, we obtained the main asymptotic term of the infinite eigenvalue sequence which accumulates at ${\mathcal E}$ from below. \\
In the present article we undertake a related program in the case where $\theta' = \beta - \epsilon$ but now $\beta$ is a general $2\pi$-periodic function while $\epsilon$ decays at infinity as before. In this case the unperturbed operator $H_\beta$ is again unitarily equivalent under an appropriate Floquet--Bloch mapping to an analytically fibered operator (see below \eqref{x3}) but there are several substantial differences with respect to the case of constant $\beta$. First, apart from the unbounded gap $(-\infty, \inf \sigma(H_\beta))$ in the spectrum of $H_\beta$, there could  also exist bounded gaps. Thus, there could be several sequences of  discrete eigenvalues of $H_{\beta - \epsilon}$ which may accumulate from above (resp., from below) to a lower (resp., to an upper) edge of a gap in $\sigma(H_\beta)$. Moreover, the bounded gaps in $\sigma(H_\beta)$ are surrounded from both sides by regions of the essential spectrum which makes the investigation of the discrete spectrum of $H_{\beta-\eps}$ more difficult in comparison with the one lying below $\inf \sigma(H_\beta)$, taking into account in particular, that the perturbation $H_{\beta-\eps} - H_\beta$ is a second-order differential operator. Further, in \cite{BKRS} it was found that the effective Hamiltonian which models the asymptotic behaviour of the discrete spectrum of $H_{\beta - \epsilon}$ near the edges of the essential one, has the form
 \bel{x4}
-\mu \frac{d^2}{dx^2} - \eta \epsilon(x), \quad x \in \re,
 \ee
where $\mu > 0$ is a constant related to the so-called {\em effective mass} while $\eta \geq 0$ is another constant which depends explicitly on $\beta$ and the geometry of $\omega$.
If $\epsilon$ decays regularly enough at infinity, the asymptotic behaviour of the discrete spectrum of the operator \eqref{x4} is well known, and generically is of semiclassical nature (see e.g.
\cite[Theorem XIII.82]{RS4} for the generic case, and \cite{KS1} for the corrections to the semiclassical behaviour in the border-line case). In the present paper we find that the effective Hamiltonian which governs the asymptotics of the discrete spectrum of $H_{\beta-\epsilon}$ near a given edge of a gap in $\sigma(H_\beta)$,  can be written as a finite orthogonal sum of operators of the form
 \bel{x4a}
-\mu \frac{d^2}{dx^2} - \eta_{\rm per}(x) \epsilon(x), \quad x \in \re,
 \ee
 where $\mu > 0$ again is a constant related to the effective mass at the edge, but $\eta_{\rm per}$ is a periodic, generically non constant function which depends on $\beta$ and $\omega$. Note that even if $\epsilon$ decays regularly at infinity, the product $\eta_{\rm per}\, \epsilon$ has an irregular decay due to the oscillations of $\eta_{\rm per}$. Thus, the eigenvalue asymptotics for operators like \eqref{x4a} could be of independent interest. Multidimensional Schr\"odinger operators of this type have been considered in a different context in \cite{LNS, S}. \\
The article is organized as follows. In the next section we describe the spectral properties of the unperturbed operator $H_\beta$, necessary for the statement and the understanding of our main results, formulate these results, and briefly comment on them.
Their proofs can be found in Section \ref{s4}. Finally, in the Appendix we prove an auxiliary proposition concerning the spectral properties of an effective Hamiltonian of the form \eqref{x4a}.

\section{Main Results}
\label{s2} \setcounter{equation}{0}
\subsection{Spectral properties of the unperturbed operator $H_\beta$}
\label{ss21}
Assume that $\beta  \in C({\mathbb T};\re)$ where ${\mathbb T} : = \re/2\pi{\mathbb Z}$. Set ${\mathbb T}^* : = \re/{\mathbb Z}$. Define the unitary Floquet--Bloch operator
$\Phi: {\rm L}^2(\Omega) \to {\rm L}^2(\omega \times {\mathbb T} \times {\mathbb T}^*)$ by
$$
(\Phi u)(x_t,x_3,k) : = \sum_{\ell \in {\mathbb Z}} e^{-ik(x_3+ 2 \pi \ell)} u(x_t,x_3+ 2\pi \ell), \quad x_t \in \omega, \quad x_3 \in {\mathbb T}, \quad k \in {\mathbb T}^*,
$$
for, say, $u \in C(\overline{\omega}; {\mathcal S}(\re))$, where ${\mathcal S}(\re)$ denotes the Schwartz class on $\re$. Similar Floquet--Bloch operators have been used by numerous authors (see e.g. \cite{Y, BDE, F2, BP2, BP1}) within the context of the spectral analysis of periodic quantum waveguides.  We have
\bel{x3}
\Phi H_\beta \Phi^* = \int_{{\mathbb T}^*}^{\oplus} h_\beta(k) dk
\ee
where $h_\beta(k)$, $k \in {\mathbb T}^*$, is the self-adjoint operator generated in ${\rm L}^2(\omega \times {\mathbb T})$ by the closed quadratic form
$$
q[u; k] = \int_{\omega} \int_{\mathbb T}\left(|\nabla_t u|^2 + |(\beta \partial_\varphi + \partial_3 +ik)u|^2\right)dx_3 dx_t, \quad u \in {\rm H}_0^1(\omega \times {\mathbb T}).
$$
Note that
$$
q[u; k] \asymp \int_{\omega} \int_{\mathbb T}\left(|\nabla_t u|^2 + |(\partial_3 +ik)u|^2\right)dx_3 dx_t
$$
  uniformly with respect to $k \in {\mathbb T}^*$; here and in the sequel the notation $A \asymp B$ means that there exist constants $0 < c_1 \leq c_2 < \infty$ independent of $A$ and $B$, such that $c_1 A \leq B \leq c_2 A$. Evidently, the operator $h_\beta(k)$, $k \in {\mathbb T}^*$, is elliptic; since $\omega$ is  bounded, we find that the spectrum of $h_\beta(k)$ is discrete. Denote by $\left\{E_\ell(k)\right\}_{\ell \in {\mathbb N}}$ the non-decreasing sequence of the eigenvalues of $h_\beta(k)$, $k \in {\mathbb T}^*$. By the  Kato perturbation theory \cite{K}, {\em the band functions} $E_\ell$ are continuous piece-wise real analytic functions. We have
 $$
 \sigma(H_\beta) = \bigcup_{\ell \in {\mathbb N}} E_\ell({\mathbb T}^*).
 $$
 Let ${\mathcal E}_0^+ : = \inf \sigma(H_\beta) = \min_{k \in {\mathbb T}^*} E_1(k)$; evidently, ${\mathcal E}_0^+ > 0$. Then  in $\sigma(H_\beta)$ there always exists a semi-bounded gap\footnote{We call a {\em gap} in $\sigma(H_\beta)$ any open non-empty interval ${\mathcal I} \subset \re\setminus \sigma(H_\beta)$ such that $\partial {\mathcal I} \subset \sigma(H_\beta)$.} $(-\infty, {\mathcal E}_0^+)$. In contrast to the case of constant $\beta$, in $\sigma(H_\beta)$ with periodic non constant $\beta$ there could  also be bounded gaps (see e. g. \cite[Subsection 3.4]{DeO}, \cite[Subsection 4.4]{BP1}).
 Let $\left({\mathcal E}_j^-, {\mathcal E}_j^+\right)$, $j=1,\ldots,J \leq \infty$, be the disjoint bounded gaps in $\sigma(H_\beta)$; if there are no bounded gaps in $\sigma(H_\beta)$, we set $J=0$.    Then we have
    \bel{x6}
 \re\setminus\sigma(H_\beta) = \bigcup_{j=0}^J \left({\mathcal E}_j^-, {\mathcal E}_j^+\right)
    \ee
 with ${\mathcal E}_0^- : = -\infty$. Note that the value ${\mathcal E}_j^-$, $j \geq 1$, (resp.,  ${\mathcal E}_j^+$, $j \geq 0$) coincides with the maximal (resp., minimal) value of some band function $E_\ell$.
 \begin{defi} \label{d21}
 We will say that the boundary point ${\mathcal E}_j^{\pm}$ of  $\sigma(H_\beta)$ is regular if:\\
 {\rm (i)} There exists a unique band function $E_{\ell(j)}^{\pm}$ in the sequence $\left\{E_\ell\right\}_{\ell \in {\mathbb N}}$ which attains the value ${\mathcal E}_j^{\pm}$. \\
 {\rm (ii)} The function $E_{\ell(j)}^{\pm}$ attains the value ${\mathcal E}_j^{\pm}$ at finitely many points $k_{j,m}^{\pm}$, $m=1,\ldots, M^{\pm}_j$.\\
 {\rm (iii)} We have
   \bel{x25}
 \mu_{j,m}^{\pm} : = \pm \frac{1}{2} \frac{d^2E^{\pm}_{\ell(j)}}{dk^2}(k_{j,m}^{\pm}) > 0, \quad m=1,\ldots,M^{\pm}_j.
    \ee
  \end{defi}
  Note that if conditions (i) and (ii) in Definition \ref{d21} hold true, then the function $E^{\pm}_{\ell(j)}$ is analytic in a vicinity  of each point $k_{j,m}^{\pm}$, $m=1,\ldots,M_j^{\pm}$. More precisely, there exists a $\delta > 0$ such that the intervals
  \bel{x9}
  {\mathcal I}^{\pm}_{j,m} = \left(k_{j,m}^{\pm}-\delta, k_{j,m}^{\pm}+\delta\right), \quad m=1,\ldots,M_j^{\pm},
  \ee
  are disjoint, and the function $E_{\ell(j)}^{\pm}$ is real-analytic on their closures. Set
  \bel{x10}
  {\mathcal I}_j^{\pm} = \bigcup_{m=1}^{M_j^{\pm}} {\mathcal I}_{j,m}^{\pm},
  \ee
  and introduce the eigenfunctions $\psi_j^{\pm}(\bx; k)$, $\bx = (x_t,x_3) \in \omega \times {\mathbb T}$, $k \in {\mathcal I}_j^{\pm}$, such that
  \bel{x11}
  h_\beta(k) \psi_j^{\pm}(\cdot ; k) = E^\pm_{\ell(j)} \psi_j^{\pm}(\cdot ; k), \quad \int_{\omega} \int_{\mathbb T} \left|\psi_j^{\pm}(x_t, x_3 ; k)\right|^2 dx_3 dx_t =1, \quad k \in   {\mathcal I}_j^{\pm},
  \ee
  and the mappings $\overline{ {\mathcal I}_j^{\pm}} \ni k \mapsto \psi_j^{\pm}(\cdot ; k) \in D(H_\beta)$ are analytic. \\
 The following proposition shows that the set of regular edges of $\sigma(H_\beta)$ is not empty.
 \begin{pr} \label{p22}
 Assume that $\partial \omega \in C^\infty$ and $\beta \in C^\infty({\mathbb T})$. Then $E_1(0) = {\mathcal E}_0^+  = \min_{k \in {\mathbb T}^*} E_1(k)$, and we have
 $E_1(k) > E_1(0)$, for $k \in  {\mathbb T}^*$, $ k \neq 0$, as well as $E_j(0) > E_1(0)$ for $j\geq 2$. Moreover, $E''_1(0) > 0$.
 \end{pr}
 \begin{proof}
 The operator $h_\beta(0) = -\Delta_t - (\beta\partial_\varphi +\partial_3)^2$ is a strongly elliptic operator on $\omega \times {\mathbb T}$ with smooth real coefficients. Hence, its first eigenvalue is simple, i.e. $E_j(0) > E_1(0)$, $j \geq 2$. Moreover, we could choose the normalized first eigenfunction $\psi  \in C^{\infty}(\overline{\omega} \times {\mathbb T})$ of $h_\beta(0)$ to be  positive on $\omega \times {\mathbb T}$. The mini-max principle yields
 $$
 E_1(k) = \inf_{0 \neq u \in C^\infty({\mathbb T}; C_0^\infty(\omega))}\frac{\int_\omega \int_{\mathbb T}\left(|\nabla_t u|^2 + |(\beta \partial_\varphi + \partial_3 + ik)u|^2\right)dx_3 dx_t}{\int_\omega \int_{\mathbb T}|u|^2dx_3 dx_t}, \quad k \in {\mathbb T}^*.
 $$
 Changing the functional variable $u = \psi v$, and integrating by parts, we obtain
 $$
 E_1(k) - E_1(0) =
 $$
 \bel{y1}
 \inf_{0 \neq v \in C^\infty({\mathbb T}; C_0^\infty(\omega))}\frac{\int_\omega \int_{\mathbb T}\psi^2 \left(|\nabla_t v|^2 + |(\beta \partial_\varphi + \partial_3 + ik)v|^2\right)dx_3 dx_t}{\int_\omega \int_{\mathbb T}\psi^2 |v|^2dx_3 dx_t}, \quad k \in {\mathbb T}^*.
 \ee
 Further, for any $\varepsilon \in (0,1)$ we have
    $$
    \int_\omega \int_{\mathbb T}\psi^2 \left(|\nabla_t v|^2 + |(\beta \partial_\varphi + \partial_3 + ik)v|^2\right)dx_3 dx_t \geq
    $$
    \bel{y2}
    \int_\omega \int_{\mathbb T}\psi^2 \left((1-(\varepsilon^{-1}-1)c)|\nabla_t v|^2 + (1-\varepsilon) |\partial_3 v + ik v|^2\right)dx_3 dx_t
    \ee
    where $c : = \max_{x_3 \in {\mathbb T}} \beta(x_3)^2 \sup_{x_t \in \omega} |x_t|^2$. Now, \eqref{y1} and \eqref{y2} yield
    \bel{y3}
 E_1(k) - E_1(0) \geq
  (1+c)^{-1} \inf_{0 \neq v \in C^\infty({\mathbb T}; C_0^\infty(\omega))}\frac{\int_\omega \int_{\mathbb T} \psi^2 |\partial_3 v + ik v|^2dx_3 dx_t}{\int_\omega \int_{\mathbb T}\psi^2 |v|^2 dx_3 dx_t}, \quad k \in {\mathbb T}^*.
 \ee
 Let $\Psi(x_t) : = {\rm dist}\,(x_t,\partial \omega)$, $x_t \in \omega$.  Then there exist constants $0 < c_1 \leq c_2 < \infty$ such that
 \bel{y4}
 c_1 \Psi(x_t) \leq \psi(x_t,x_3) \leq c_2 \Psi(x_t), \quad x_t \in \omega, \quad x_3 \in {\mathbb T}.
 \ee
 The lower bound in \eqref{y4} can be obtained arguing as in the proof of \cite[Theorem 7.1]{DS}. The upper bound follows from the facts that
 $\psi  \in C^{\infty}(\overline{\omega} \times {\mathbb T})$, $\psi_{|\partial \omega \times {\mathbb T}} = 0$, and in a vicinity of $\partial \omega$ there exist smooth coordinates in $\overline{\omega}$ such that the variable normal to $\partial \omega$ is proportional to $\Psi(x_t)$.
 Expanding $w \in C^\infty({\mathbb T})$ in a Fourier series, we easily find that
 $$
 \inf_{0 \neq w \in C^\infty({\mathbb T})}\frac{ \int_{\mathbb T} |w' + ik w|^2 dx}{\int_{\mathbb T} |w|^2 dx} = k^2, \quad k \in {\mathbb T}^*.
 $$
 Therefore,
 \bel{y5}
 \int_{\mathbb T} |\partial_3 v(x_t, x_3) + ik v(x_t, x_3)|^2 dx_3 \geq  k^2 \, \int_{\mathbb T} |v(x_t, x_3)|^2 dx_3
 \ee
 for any $x_t \in \omega$ and $v \in C^\infty({\mathbb T}; C_0^\infty(\omega))$. Multiplying \eqref{y5} by $\Psi(x_t)^2$ and integrating with respect to $x_t \in \omega$, bearing in mind \eqref{y4}, we obtain the estimate
    \bel{y6}
    \inf_{0 \neq v \in C^\infty({\mathbb T}; C_0^\infty(\omega))}\frac{\int_\omega \int_{\mathbb T} \psi^2 |\partial_3 v + ik v|^2dx_3 dx_t}{\int_\omega \int_{\mathbb T}\psi^2 |v|^2 dx_3 dx_t} \geq \frac{c_1^2}{c_2^2} k^2, \quad k \in {\mathbb T}^*.
    \ee
    Now \eqref{y3} and \eqref{y6} yield
    \bel{y7}
    E_1(k) - E_1(0) \geq  \frac{c_1^2}{(1+c)c_2^2} \, k^2, \quad k \in {\mathbb T}^*.
    \ee
    In particular we have, $E_1(k) > E_1(0)$ for $0 \neq k \in {\mathbb T}^*$. Since $E_1$ is analytic in a neighbourhood of $k=0$, we find that \eqref{y7} also  implies $E_1'(0) = 0$, $E_1''(0) > 0$.
    \end{proof}
    {\em Remark}: The assumptions $\partial \omega \in C^{\infty}$ and $\beta \in C^{\infty}({\mathbb T})$ of Proposition \ref{p22} are too restrictive; we impose them for the sake of simplicity of the proof.\\

 Let us now comment on the validity in general of conditions (i) -- (iii) in Definition \ref{d21}. It is well known that in the case of 1D Schr\"odinger operators with $2\pi$-periodic  potentials (Hill operators), the analogue of condition (i) is always fulfilled (see e.g. \cite[Theorem XIII.89]{RS4}). The results of \cite{KR} imply that generically this is also the case for multidimensional Schr\"odinger operators with periodic  potentials. It is quite likely that the methods of \cite{KR} could be successfully applied in order to show that condition (i) in Definition \ref{d21} is generically valid.\\
 Further, condition (ii) would immediately follow from condition (i) if we know that the band function $E_{\ell(j)}^{\pm}$ is not constant on any interval
 of positive length. On the other hand, the non constancy of $E_{\ell(j)}^{\pm}$ would follow from the absolute continuity of $\sigma(H_\beta)$, which however has not been proven yet in maximal generality. Probably, the most general results concerning the absolute continuity of the spectrum for periodic quantum waveguides, are contained in \cite{F2}; reduced to the special case of $H_\beta$, these results imply that $\sigma(H_\beta)$ is purely absolutely continuous under the (technical) assumption that $\beta$ is an {\em odd} sufficiently regular periodic function of $x_3$. Essentially less general result could be found in \cite{BDE} where it is shown that for each $E>0$ there exists $\varepsilon > 0$ such that $0 \in \omega$  and  ${\rm diam}\,\omega < \varepsilon$ imply that $\sigma(H_\beta)$  on $(-\infty,E)$ is purely absolutely continuous.\\
 Let us comment briefly on the possible number of points $M_j^{\pm}$ at which the band function $E_{\ell(j)}^{\pm}$ attains its extremal value ${\mathcal E}_j^{\pm}$. It is well known that in the case of the Hill operator, the band functions $E_{2j-1}$ (resp., $E_{2j}$) attain their minimal value at $k=0$ and their  maximal value at $k=1/2$ (resp., their minimal value at $k=1/2$ and their  maximal value at $k=0$). This phenomenon is related, in particular, to the transformation properties of the fiber operator under complex conjugation. Our fiber operator $h_\beta(k)$ is also anti-unitarily equivalent to $h_\beta(-k)$ under complex conjugation. Hence, $E_{\ell(j)}^{\pm}$ could attain the value ${\mathcal E}_j^{\pm}$ at a {\em single} point $k \in {\mathbb T}^*$ only if $k=0$ or $k=1/2$; if $E_{\ell(j)}^{\pm}(k) = {\mathcal E}_j^{\pm}$ at some $k \in (0,1/2)$, then $E_{\ell(j)}^{\pm}(-k) = {\mathcal E}_j^{\pm}$ as well, and in this case $k$ and $-k$ are distinct points of ${\mathbb T}^*$.  In principle, our band functions $E_j$ could attain their minimal and maximal values at several points of the dual torus ${\mathbb T}^*$ (see \cite{BP2} for an example concerning a particular 2D periodic waveguide, as well as \cite[Example 4.4]{BP1} concerning a 3D waveguide with weak periodic twisting), which is reflected  in condition (ii) of Definition \ref{d21}.  \\
 Finally, the analogue of condition (iii) in Definition \ref{d21} for Hill operators is always fulfilled (see e.g. the proof of \cite[Theorem XIII.89 (e)]{RS4}). In the case of multidimensional Schr\"odinger operators with periodic electric potentials, the analogue of this condition is known to hold true at the infimum of the spectrum (see \cite{KS}) but, as far as the author is informed, there is no general proof that it holds at the edges of eventual {\em bounded} gaps in the spectrum. Note that in our case conditions (i) and (ii) imply that for each $k_{j,m}^{\pm}$ there exists $q \in {\mathbb N}$ such that the derivatives of $E_{\ell(j)}^{\pm}$ at $k_{j,m}^{\pm}$ of order $1, \ldots, 2q-1$, vanish, but the derivative of order $2q$ does not. The proofs of our main results could be easily extended
 to the case of degenerate extrema, i.e. the case $q>1$; we do not include these quite straightforward but tedious extensions just because we do not dispose of examples that such degenerate extrema could in fact occur.

 \subsection{Statement of main results}
\label{ss22}
 Let $T$ be a self-adjoint operator in a Hilbert space, and ${\mathcal I} \subset \re$ be an interval. Set
 \bel{24}
 N_{\mathcal I}(T) : = {\rm rank}\,\one_{\mathcal I}(T)
 \ee
 where $\one_{\mathcal I}(T)$ is the spectral projection of  $T$ corresponding to  ${\mathcal I}$. Thus, if
  ${\mathcal I} \cap \sigma_{\rm ess}(T) = \emptyset$, then $N_{\mathcal I}(T)$ is just the number of the (discrete) eigenvalues of $T$, lying on the interval ${\mathcal I}$, and counted with their multiplicities. \\
  Assume that $\beta \in C({\mathbb T};\re)$, $\epsilon \in C(\re; \re) \cap {\rm L}^{\infty}(\re)$, $\lim_{|x| \to \infty} \epsilon(x) = 0$. Then the resolvent difference $H_\beta^{-1} - H_{\beta-\eps}^{-1}$ is a compact operator, and hence $\sigma_{\rm ess}(H_\beta) = \sigma_{\rm ess}(H_{\beta - \epsilon})$. Therefore, \eqref{x6} implies $$
  \re\setminus\sigma_{\rm ess}(H_{\beta-\eps}) = \bigcup_{j=0}^J \left({\mathcal E}_j^-, {\mathcal E}_j^+\right).
  $$
  Put
  $$
  {\mathcal N}_0^+(\lambda) = N_{(-\infty, {\mathcal E}_0^+ - \lambda)}(H_{\beta-\eps}), \quad \lambda > 0.
  $$
  Fix ${\mathcal E} \in  \left({\mathcal E}_j^-, {\mathcal E}_j^+\right)$, $j\geq 1$, and set
  $$
  {\mathcal N}_j^-(\lambda) = N_{({\mathcal E}_j^- + \lambda, {\mathcal E})}(H_{\beta-\eps}), \quad \lambda \in (0, {\mathcal E}-{\mathcal E}_j^-),
  $$
  $$
  {\mathcal N}_j^+(\lambda) = N_{({\mathcal E}, {\mathcal E}_j^+ - \lambda)}(H_{\beta-\eps}), \quad \lambda \in (0, {\mathcal E}_j^+-{\mathcal E}).
  $$
  Assume that the edge point ${\mathcal E}_j^{\pm}$ is regular (see Definition \ref{d21}). For $x_3 \in {\mathbb T}$ and $m =1,\ldots,M_j^{\pm}$,  introduce the functions
  \bel{x1}
  \eta_{j,m}^{\pm}(x_3) : =
  2{\rm Re}\,\int_\omega \overline{\partial_\varphi\psi_j^{\pm}(x_t, x_3; k_{j,m}^{\pm})}
  \left(\beta(x_3)\partial_\varphi + \partial_3 + ik_{j,m}^{\pm}\right) \psi_j^{\pm}(x_t, x_3; k_{j,m}^{\pm}) dx_t,
  \ee
  and their mean values
  $$
  \langle  \eta_{j,m}^{\pm} \rangle : = \frac{1}{2\pi} \int_{\mathbb T}  \eta_{j,m}^{\pm}(x) dx.
  $$
  For $n \in {\mathbb Z}_+$ and $\alpha > 0$   set
    \bel{x8}
  {\mathcal S}_{n,\alpha}(\re) : = \left\{u \in C^n(\re ; \re) \, | \, |u^{(\ell)}(x)| \leq c_\ell (1+|x|)^{-\alpha - \ell}, \; x \in \re, \; \ell=0,\ldots,n\right\}.
    \ee
  Denote by ${\mathcal S}^+_{n,\alpha}(\re)$ the class of functions $u \in {\mathcal S}_{n,\alpha}(\re)$ for which there exist constants $C>0$ and $R>0$ such that $u(x) \geq C|x|^{-\alpha}$ for $|x| \geq R$.
  Now we are in position to formulate our main result.
  \begin{theorem} \label{th1}
  Let $\beta \in C^4({\mathbb T})$, and $\left({\mathcal E}_j^-, {\mathcal E}_j^+\right)$, $j \geq 0$, be a gap in $\sigma(H_\beta)$. Assume that the edge point ${\mathcal E}_j^{\pm}$ is regular.\\
  {\rm (i)} Let $\alpha \in (0,2)$, $\eps \in {\mathcal S}_{4,\alpha}^+(\re)$. Assume  that there exists at least one $m=1,\ldots,M_j^{\pm}$, such that $\pm \langle \eta_{j,m}^{\pm} \rangle > 0$. Then we have
  $$
  {\mathcal N}^{\pm}_j(\lambda) =
  \frac{1}{2\pi} \sum_{m=1}^{M_j^{\pm}} \left|\left\{(x,k) \in T^*\re \, |\, \mu_{j,m}^{\pm}k^2 \mp 2\pi \langle \eta^{\pm}_{j,m} \rangle \eps(x) < -\lambda \right\}\right|(1+o(1)) =
  $$
  \bel{21}
  \frac{1}{\pi} \sum_{m=1}^{M_j^{\pm}} \left(\mu_{j,m}^{\pm}\right)^{-1/2} \int_\re \left(\pm 2\pi \langle \eta^{\pm}_{j,m} \rangle \eps(x) - \lambda\right)_+^{1/2}dx \, (1+o(1)) \asymp \lambda^{\frac{1}{2} - \frac{1}{\alpha}}, \quad \lambda \downarrow 0,
  \ee
   where $|\cdot |$ denotes the Lebesgue measure. In particular, the fact that ${\mathcal N}^{\pm}_j(\lambda)$ grows unboundedly as $\lambda \downarrow 0$ implies that there exists a sequence of discrete eigenvalues of the operator $H_{\beta - \eps}$ which accumulates at ${\mathcal E}_j^{\pm}$.
     If, on the contrary, we have $\pm \langle \eta_{j,m}^{\pm} \rangle < 0$ for all $m=1,\ldots,M_j^{\pm}$, then
  \bel{22}
  {\mathcal N}^\pm_j(\lambda) = O(1), \quad \lambda \downarrow 0,
  \ee
  i.e. the discrete spectrum of $H_{\beta - \eps}$ does not accumulate at ${\mathcal E}_j^{\pm}$.\\
  {\rm (ii)} Let $\alpha \in (0,2)$, $\eps \in {\mathcal S}_{4,\alpha}(\re)$. Assume that $\pm \langle \eta_{j,m}^{\pm} \rangle \leq  0$ for all $m$, and $\langle \eta_{j,m}^{\pm} \rangle =  0$ for some $m=1,\ldots,M_j^{\pm}$.
  Then for each $\kappa >0$ we have
  \bel{x2}
  {\mathcal N}^\pm_j(\lambda) = O(\lambda^{\frac{1}{2}-\frac{1}{2\alpha} -\kappa}), \quad \lambda \downarrow 0,
  \ee
  if $\alpha \in (0,1]$, while \eqref{22} holds true if $\alpha \in (1,2)$. \\
  {\rm (iii)} Let $\alpha = 2$, $\eps  \in {\mathcal S}_{4,2}(\re)$. Suppose moreover that there exists a finite limit
 $L : = \lim_{|x| \to \infty}x^2 \eps(x)$. Then we have
 $$
 \lim_{\lambda \downarrow 0} |\ln{\lambda}|^{-1} {\mathcal N}_j^{\pm}(\lambda) = \frac{1}{\pi} \sum_{m = 1}^{M_j^{\pm}} \left(\frac{\pm 2\pi \langle \eta_{j,m}^{\pm} \rangle L}{\mu_{j,m}^{\pm}} - \frac{1}{4}\right)_+^{1/2}.
 $$
 If, moreover, $\pm 8\pi \langle \eta^{\pm}_{j,m} \rangle L  < \mu_{j,m}^{\pm}$ for all $m=1,\ldots,M_j^{\pm}$, then \eqref{22} holds true. \\
  {\rm (iv)} Let $\alpha > 2$, $\eps \in {\mathcal S}_{4,\alpha}(\re)$. Then \eqref{22} holds true again.
\end{theorem}
{\em Remark}: As mentioned in the Introduction, the case of a constant $\beta$ was considered in \cite{BKRS}; in this case our Theorem \ref{th1} reduces, after minor modifications of the assumptions, to \cite[Theorem 4.4]{BKRS}. Note that if $\beta$ is constant, then $(-\infty, {\mathcal E}_0^+)$ is the only gap in $\sigma(H_\beta)$, the value $ {\mathcal E}_0^+$ is attained only by the band function $E_1$ at the unique point $k_{0,1}^+ = 0$, and $E_1''(0) > 0$
(see \cite[Theorem 3.1]{BKRS}). Moreover, the eigenfunction $\psi_0^+(\cdot;0)$ is real valued and independent of $x_3$, so that we have
$$
\eta_{0,1}^+ = \langle \eta_{0,1}^+ \rangle = 2\beta \int_\omega \left( \partial_\varphi \psi_0^+(x_t;0)\right)^2 dx_t.
$$
It could be shown that the analogue of Theorem \ref{th1} (ii)  could be then strengthened, namely $\eta_{0,1}^+ = 0$ implies that the spectrum of $H_{\beta-\epsilon}$ is purely essential for any reasonable decaying $\eps$, so that ${\mathcal N}_0^+(\lambda) = 0$ for any $\lambda > 0$.\\
\subsection{Comments on the main results}
\label{ss23}
Introduce the operator
$$
{\mathcal H}_j^{\pm} : = \bigoplus_{m=1}^{M_j^{\pm}} \left(- \mu_{j,m}^{\pm} \frac{d^2}{dx^2} \mp 2\pi \langle \eta^{\pm}_{j,m} \rangle \eps \right),
$$
self-adjoint in ${\rm L}^2(\re; {\mathbb C}^{M_j^{\pm}})$. Proposition \ref{p21} below shows that ${\mathcal H}_j^{\pm}$ could be considered as the effective Hamiltonian which governs the asymptotic behaviour of the discrete spectrum of $H_{\beta-\eps}$ near the regular spectral edge ${\mathcal E}_j^{\pm}$. More precisely,
    \bel{x60}
{\mathcal N}_j^{\pm}(\lambda) \sim N_{(-\infty, -\lambda)}\left({\mathcal H}_j^{\pm}\right), \quad \lambda \downarrow 0.
    \ee
Asymptotic relation \eqref{x60} means that:
\begin{itemize}
\item We have
$$
\lim_{\lambda \downarrow 0} \frac{{\mathcal N}_j^{\pm}(\lambda)}{N_{(-\infty, -\lambda)}\left({\mathcal H}_j^{\pm}\right)} = 1
$$
if $N_{(-\infty, -\lambda)}\left({\mathcal H}_j^{\pm}\right)$ grows unboundedly as $\lambda \downarrow 0$ (except, possibly, for the case where \eqref{x2} holds true; then ${\mathcal N}_j^{\pm}(\lambda)$ and $N_{(-\infty, -\lambda)}\left({\mathcal H}_j^{\pm}\right)$ admit upper bounds of the same order);
\item The function ${\mathcal N}_j^{\pm}(\lambda)$ remains bounded as $\lambda \downarrow 0$ if the same is true for $N_{(-\infty, -\lambda)}\left({\mathcal H}_j^{\pm}\right)$ (again except, possibly, for the case where \eqref{x2} holds true).
\end{itemize}
Let us now formulate Proposition \ref{p21}. Let $\eta \in C({\mathbb T}; \re)$. Set
$$
\eta_\ell : = \frac{1}{\sqrt{2\pi}} \int_0^{2\pi} \eta(x) e^{-i\ell x} dx, \quad \ell \in {\mathbb Z},
\quad \langle \eta \rangle : = \frac{1}{2\pi} \int_0^{2\pi}  \eta(x) dx.
$$
Let $\mu > 0$, $\eps \in {\rm L}^{\infty}(\re; \re)$. Introduce the operator
$$
h_{\rm eff} : = -\mu \frac{d^2}{dx^2} - \eta(x) \eps(x), \quad x \in \re,
$$
self-adjoint in ${\rm L}^2(\re)$.
\begin{pr} \label{p21}
Let $\eta \in C({\mathbb T}; \re)$. Assume that $\left\{\eta_\ell\right\}_{\ell \in {\mathbb Z}} \in \ell^1({\mathbb Z})$.\\
{\rm (i)} Let $\alpha \in (0,2)$, $\eps \in {\mathcal S}_{4,\alpha}^+(\re)$. Assume that $\langle \eta \rangle > 0$. Then we have
  $$
  N_{(-\infty, -\lambda)}(h_{\rm eff}) =
  \frac{1}{2\pi} \left|\left\{(x,k) \in T^*\re \, |\, \mu k^2 - \langle \eta \rangle \eps(x) < -\lambda \right\}\right|(1+o(1)) =
  $$
  $$
  \frac{1}{\pi \sqrt{\mu}}   \int_\re \left(\langle \eta \rangle \eps(x) - \lambda\right)_+^{1/2}dx \, (1+o(1)) \asymp \lambda^{\frac{1}{2} - \frac{1}{\alpha}}, \quad \lambda \downarrow 0.
  $$
  If, on the contrary, $\langle \eta \rangle < 0$, then
  \bel{x35}
  N_{(-\infty, -\lambda)}(h_{\rm eff}) = O(1), \quad \lambda \downarrow 0.
  \ee
  {\rm (ii)} Let $\alpha \in (0,2)$, $\eps  \in {\mathcal S}_{4,\alpha}(\re)$. Assume that $\langle \eta \rangle =  0$
  Then for each $\kappa >0$ we have
  $$
  N_{(-\infty, -\lambda)}(h_{\rm eff}) = O(\lambda^{\frac{1}{2}-\frac{1}{2\alpha} -\kappa}), \quad \lambda \downarrow 0,
  $$
  if $\alpha \in (0,1]$, while \eqref{x35} holds true if $\alpha \in (1,2)$. \\
  {\rm (iii)} Let $\alpha = 2$, $\eps  \in {\mathcal S}_{4,2}(\re)$. Suppose moreover that there exists $L \in \re$ such that
 $ \lim_{|x| \to \infty}x^2 \eps(x) = L$. Then we have
 $$
 \lim_{\lambda \downarrow 0} |\ln{\lambda}|^{-1} N_{(-\infty, -\lambda)}(h_{\rm eff}) = \frac{1}{\pi}  \left(\frac{\langle \eta\rangle L}{\mu}  - \frac{1}{4}\right)_+^{1/2}.
 $$
 If, moreover, $4 \langle \eta \rangle L  < \mu$, then \eqref{x35} holds true. \\
  {\rm (iv)} Let $\alpha > 2$, $\eps  \in {\mathcal S}_{0,\alpha}(\re)$. Then \eqref{x35} holds true again.
\end{pr}
Possibly,  Proposition \ref{p21} is known to the experts. However, we could not find it in the literature and that is why we include its proof in the Appendix. The proposition could be of independent interest due, in particular, to the non semiclassical nature of some of its results. Proposition \ref{p21} admits far going extensions to multidimensional Schr\"odinger operators; hopefully, we will consider them in a future work.

\section{Proof of the Main Result}
\label{s4} \setcounter{equation}{0}

\subsection{Auxiliary results}
\label{ss41}
This subsection contains auxiliary results needed for the proof of Theorem \ref{th1}.\\
Let $X_j$, $j=1,2$, be two separable Hilbert spaces. We denote by $S_\infty(X_1,X_2)$ the class of linear compact
operators $T: X_1 \to X_2$. If $X_1 = X_2 = X$, we write $S_\infty(X)$ instead of $S_\infty(X,X)$. Let $T = T^* \in S_\infty(X)$. For $s>0$ set
$$
n_{\pm}(s;T) : = N_{(s,\infty)}(\pm T)
$$
(see \eqref{24}); thus, $n_+(s;T)$ (resp., $n_-(s;T))$ is just the number of the eigenvalues of $T$ larger than $s$ (resp., smaller than $-s$), and counted with the multiplicities. If $T_j = T_j^* \in S_\infty(X)$, $j=1,2$, then the Weyl inequalities
    \bel{41}
    n_{\pm}(s_1 + s_2; T_1 + T_2) \leq n_{\pm}(s_1;T_1) + n_{\pm}(s_2;T_2)
    \ee
    hold for $s_j>0$, $j=1,2$, (see e.g. \cite[Theorem 9,  Section 2, Chapter 9]{BSo2}). For $T \in S_\infty(X_1,X_2)$ and $s>0$ put
    \bel{42}
    n_*(s; T) : = n_+(s^2; T^*T).
    \ee
    Thus, $n_*(s;T)$ is the number of the singular values of $T$ larger than $s$, and counted with the multiplicities. If $T_j \in S_\infty(X_1, X_2)$, and $s_j>0$, $j=1,2$, then the Ky Fan inequalities
    \bel{43}
    n_*(s_1 + s_2; T_1 + T_2) \leq n_*(s_1;T_1) + n_*(s_2;T_2)
    \ee
    hold true (see e.g. \cite[Eq. (17), Section 1, Chapter 11)]{BSo2}).
    The following lemma contains spectral estimates for finite-rank and bounded perturbations.
    \begin{lemma} \label{l43} Let $-\infty < a < b <\infty$, $T=T^*$.\\
    {\rm (i) (\cite[Theorem 3, Section 3, Chapter 9]{BSo2})} Let $S= S^*$ be bounded and ${\rm rank}\,S < \infty$. Then we have
    $$
    N_{(a,b)}(T) - {\rm rank}\,S \leq N_{(a,b)}(T + S) \leq N_{(a,b)}(T) + {\rm rank}\,S.
    $$
    {\rm (ii) (\cite[Lemma 3, Section 4, Chapter 9]{BSo2})} Let $S= S^*$ be bounded and $\sigma(S) \subset [a_1,b_1]$. Then we have
    $$
    N_{(a,b)}(T) \leq N_{(a+a_1,b+b_1)}(T + S).
    $$
    \end{lemma}
    Further, we recall an abstract version of the Birman-Schwinger principle,  suitable for our purposes.
    \begin{lemma} \label{l41}
    {\rm (\cite[Lemma 1.1]{bir})} Let $T = T^* \geq 0$, and let $S = S^*$ be relatively compact in the sense of the quadratic forms with respect to $T$. Then we have
    $$
    N_{(-\infty,  - \lambda)}(T - rS) =
    n_+(r^{-1}; (T + \lambda)^{-1/2} S (T + \lambda)^{-1/2})
    $$
    for any $r>0$ and $\lambda > 0$ .
    \end{lemma}
    Our next lemma contains well known  results on the asymptotic behaviour
of the discrete spectrum for 1D Schr\"odinger operators.

\begin{lemma} \label{l42}
Assume that $\mu > 0$. \\
{\rm (i)} Let $V \in {\mathcal S}^+_{1,\alpha}(\re)$ with $\alpha \in (0,2)$.
Then we have
$$
N_{(-\infty,-\lambda)}\left(-\mu \frac{d^2}{dx^2} - V\right) = \frac{1}{\pi \sqrt{\mu}} \int_\re(V(x)-\lambda)_+^{1/2} dx \, (1+o(1)) \asymp \lambda^{\frac{1}{2} - \frac{1}{\alpha}}, \quad \lambda \downarrow 0.
$$
(ii) Assume that $V \in {\mathcal S}_{0,2}(\re)$ and  there exists a finite limit $L : = \lim_{|x| \to \infty}x^2V(x)$. Then
$$
\lim_{\lambda\downarrow 0}\,  |\ln{\lambda}|^{-1} N_{(-\infty,-\lambda)}\left(-\mu \frac{d^2}{dx^2} - V\right)  = \frac{1}{\pi} \left(\frac{L}{\mu} -
\frac{1}{4}\right)_+^{1/2}.
    $$
If, moreover, $4L < \mu$, then
    \bel{46}
N_{(-\infty,-\lambda)}\left(-\mu \frac{d^2}{dx^2} - V\right) = O(1), \quad \lambda
\downarrow 0.
    \ee
    (iii) Suppose $V \in {\mathcal S}_{0,\alpha}(\re)$ with $\alpha \in (2,\infty)$. Then \eqref{46} holds true again.
\end{lemma}
The first part of the lemma is a special case of \cite[Theorem XIII.82]{RS4},
the proof of the second part is contained in \cite{KS1}, while the
third part follows from the result of \cite[Problem 22, Chapter XIII]{RS4}.\\
The last lemma in this subsection concerns  the Fourier transform of a function $u \in {\mathcal S}_{n,\alpha}(\re)$. For $u \in {\mathcal S}(\re)$, introduce its Fourier transform
    \bel{x27}
({\mathcal F}u)(k) = \hat{u}(k) : = (2\pi)^{-1/2} \int_\re e^{-ixk} u(x) dx, \quad k \in \re.
    \ee
Whenever necessary, the Fourier transform is extended by duality to the dual Schwartz class ${\mathcal S}'(\re)$. We will use the same notations for the partial Fourier transform with respect to $x_3 \in \re$ in the case $u = u(x_t,x_3)$, $(x_t,x_3) \in \Omega$.

\begin{lemma} \label{l44} Assume that  $u \in {\mathcal S}_{n,\alpha}(\re)$, $n \in {\mathbb N}$, $\alpha > 0$. Then $\hat{u} \in C^{n-1}(\re\setminus\{0\})$, and there exists a constant $C$ such that
\bel{449}
\sup_{|k| \geq \kappa} |\hat{u}^{(n-1)}(k)| \leq C \kappa^{-2n+1}
\ee
for each $\kappa \in (0,1]$.
\end{lemma}
\begin{proof}
We have
\bel{450}
\hat{u}^{(n-1)}(k) = \sum_{s=0}^{n-1} \binom{n-1}{s} \frac{d^{n-1-s}}{dk^{n-1-s}}(k^{-n}) \frac{d^{s}}{dk^{s}}(k^n \hat{u}(k)), \quad k \in \re\setminus\{0\}.
 \ee
 Moreover,
 \bel{451a}
 \frac{d^{s}}{dk^{s}}(k^n \hat{u}(k)) = \frac{i^{n-s}}{(2\pi)^{1/2}} \int_\re e^{-ikx} x^s u^{(n)}(x) dx, \quad k \in \re.
 \ee
 Therefore, by \eqref{x8},
  \bel{451}
  \sup_{k \in \re} \left|\frac{d^{s}}{dk^{s}}(k^n \hat{u}(k))\right| \leq  \frac{c_n}{(2\pi)^{1/2}} \int_\re  |x|^s (1 + |x|)^{-\alpha -n} dx < \infty, \quad s=0,\ldots,n-1, \quad \alpha >0.
  \ee
  Now \eqref{450} -- \eqref{451} yield $\hat{u} \in C^{n-1}(\re\setminus\{0\})$ and estimate \eqref{449}.
\end{proof}
\subsection{Projection to the spectral edge of $H_\beta$}
\label{ss42}
In this section we fix the gap $\left({\mathcal E}_j^-, {\mathcal E}_j^+\right)$ in $\sigma(H_\beta)$, choose the edge ${\mathcal E}_j^-$ or ${\mathcal E}_j^+$ which is supposed to be regular, and restrict the analysis to the spectral subspace of the unperturbed operator $H_\beta$ which corresponds to a small vicinity of this edge. Set
$$
\pi^{\pm}_{j}(k) : = \left|\psi_j^{\pm}(\cdot;k)\rangle \langle \psi_j^{\pm}(\cdot;k)\right|, \quad k \in {\mathcal I}_j^{\pm},
$$
(see \eqref{x11} for the definition of $\psi_j^{\pm}$), and
$$
{\mathcal P}_j^{\pm} := \int_{{\mathcal I}_j^{\pm}}^{\oplus} \pi^{\pm}_{j}(k) dk,
\quad P_j^{\pm} : = \Phi^* {\mathcal P}_j^{\pm} \Phi, \quad Q_j^{\pm} : = I -  P_j^{\pm}.
$$
Thus, $P_j^{\pm}$ and $Q_j^{\pm}$ are orthogonal projections in ${\rm L}^2(\Omega)$. Since they commute with $H_\beta^{-1}$, they leave invariant
$D(H_\beta) = D(H_{\beta - \eps}) = {\rm H}^2(\Omega) \cap {\rm H}^1_0(\Omega)$.
Let us recall now that
\bel{x61}
H_{\beta-\eps} = H_\beta + 2\beta \eps \partial_\varphi^2 + 2 \eps \partial_\varphi \partial_3 - \eps^2 \partial_\varphi^2 + \eps'\partial_\varphi.
\ee
In particular, the perturbation $H_{\beta-\eps} - H_\beta$ is a second-order differential operator. The spectral properties for second-order localized perturbations of second-order elliptic operators were considered in \cite{AADH} in a different context.
Further, \eqref{x61} implies
    \bel{47}
    H_{\beta-\eps} = P_j^{\pm}  H_{\beta-\eps} P_j^{\pm} + Q_j^{\pm}  H_{\beta-\eps} Q_j^{\pm}
    +  \sum_{i=1}^4 \left(P_j^{\pm}  f_i L_i Q_j^{\pm} + Q_j^{\pm}  f_i L_i P_j^{\pm}\right)
    \ee
    where
    $$
    f_1 = 2\beta \eps, \quad f_2 = 2\eps, \quad f_3 = -\eps^2, \quad f_4 = \eps',
    $$
    $$
    L_1 = L_3 = \partial_\varphi^2, \quad L_2 = \partial_\varphi \partial_3, \quad L_4 = \partial_\varphi.
    $$
    Let us now write $I = H_\beta^2 H_\beta^{-2}$, then commute $f_i$ with appropriate powers of $H_\beta^{-1}$. Taking into account that
    $[f_i, H_\beta^{-1}] = H_\beta^{-1} (f_i'' + 2f_i' D) H_\beta^{-1} $ where $D : = \beta \partial_\varphi + \partial_3$, we find that
    \bel{47a}
    P_j^{\pm} f_i L_i Q_j^{\pm} = \sum_{n=0}^{n_i} \sum_{r=1}^{r_{n,i}} P_j^{\pm} H_\beta^n g_{i,n,r} H_\beta^{-1/2} K_{i,n,r} Q_j^{\pm}, \quad i=1,\ldots,4,
    \ee
    where $g_{i,n,r}$ are the multipliers by decaying functions of $x_3$, and $K_{i,n,r}$ are bounded operators in ${\rm L}^2(\Omega)$. Let us define explicitly the functions $g_{i,n,r}$ and the operators
    $K_{i,n,r}$. Fix $i=1,2,3$. Then $n_i=2$, $r_{0,i} =2$,  and
    $$
    g_{i,0,1} = f_i^{(iv)}, \quad   K_{i,0,1} = H_\beta^{-1/2}(H_\beta^{-1} - 4DH_\beta^{-1}DH_\beta^{-1}) L_i,
    $$
    $$
    g_{i,0,2} = f_i^{'''},  \quad K_{i,0,2} = 2H_\beta^{1/2} (DH_\beta^{-1} + H_\beta^{-1}D - 4 DH_\beta^{-1}DH_\beta^{-1}D) H_\beta^{-1}L_i,
    $$
    $r_{1,i} =2$,  and
    $$
    g_{i,1,1} = f_i^{''}, \quad K_{i,1,1} = -2 H_\beta^{-1/2}(I - 2 DH_\beta^{-1}D)H_\beta^{-1}L_i,
    $$
    $$
    \quad  g_{i,1,2} = f_i^{'},   \quad K_{i,1,2} = -2H_\beta^{1/2} (DH_\beta^{-1} + H_\beta^{-1}D) H_\beta^{-1}L_i,
    $$
    while $r_{2,i} =1$,  and
    $$
    g_{i,2,1} = f_i, \quad K_{i,2,1} = H_\beta^{-3/2} L_i.
    $$
    Finally, if $i=4$, then $n_4 = 1$, $r_{0,4} = 2$, and
     $$
    g_{4,0,1} = f_4^{'''}, \quad K_{4,0,1} = 2H_\beta^{-1/2}DH_\beta^{-1} L_4,
    $$
    $$
    g_{4,0,2} = f_4^{''},
       \quad
    K_{4,0,2} = -(H_\beta^{-1/2} - 4H_\beta^{1/2}DH_\beta^{-1}DH_\beta^{-1}) L_4,
    $$
    while $r_{1,4} =2$, and
    $$
    g_{4,1,1} = f_4', \quad K_{4,1,1} = -2H_\beta^{-1/2}D H_\beta^{-1} L_4, \quad g_{4,1,2} = f_4, \quad K_{4,1,2} = H_\beta^{-1/2} L_4. $$
    Hence, \eqref{47a} implies that for any $\nu \in (0,1)$ we have
     \bel{48}
     \sum_{i=1}^4 \left(P_j^{\pm}  f_i L_i Q_j^{\pm} + Q_j^{\pm}  f_i L_i P_j^{\pm}\right) = 2 {\rm Re} \sum_{i=1}^4 \sum_{n=0}^{n_i} \sum_{r=1}^{r_{n,i}} P_j^{\pm} H_\beta^n |g_{i,n,r}|^{\frac{1+\nu}{2}} S_{i,n,r}(\nu)  Q_j^{\pm},
    \ee
    where $S_{i,n,r}(\nu) : = {\rm sign} \, g_{i,n,r} |g_{i,n,r}|^{\frac{1-\nu}{2}} H_\beta^{-1/2} K_{i,n,r}$.
   \begin{lemma} \label{l401}
    Let $\partial \omega \in C^2$, $\beta \in C^4({\mathbb T};\re)$, $\eps \in {\mathcal S}_{4,\alpha}(\re)$, $\alpha > 0$.
    Then the operators $S_{i,n,r}(\nu)$ with $\nu \in (0,1)$ are compact in ${\rm L}^2(\Omega)$.
    \end{lemma}
    \begin{proof}
    Since the operators $K_{i,n,r}$ are bounded, and $\sup_{x \in \re} |g_{i,n,r}| (1 + |x|)^{\alpha} < \infty$, it suffices to show that the operator
    $(1 + |x_3|)^{-\kappa} H_\beta^{-1/2}$ with $\kappa = \alpha(1-\nu)/2 >0$
    is compact in ${\rm L}^2(\Omega)$. Evidently, the operator $H_0^{1/2} H_\beta^{-1/2}$ is bounded in ${\rm L}^2(\Omega)$, so that it suffices to prove that the operator $(1 + |x_3|)^{-\kappa} H_0^{-1/2}$
    is compact in ${\rm L}^2(\Omega)$. Expanding the function $u \in {\rm L}^2(\Omega)$ with respect to the eigenfunctions of the Dirichlet Laplacian $-\Delta_t$, self-adjoint in ${\rm L}^2(\omega)$, we find that $(1 + |x_3|)^{-\kappa} H_0^{-1/2}$ is unitarily equivalent to the orthogonal sum
    \bel{411}
    \bigoplus_{\ell \in {\mathbb N}}(1 + |x|)^{-\kappa} \left(-\frac{d^2}{dx^2} + \lambda_\ell\right)^{-1/2},
    \ee
    self-adjoint in $ \ell^2({\mathbb N};{\rm L}^2(\re))$; here, $\left\{\lambda_\ell\right\}_{\ell \in {\mathbb N}}$ is the non decreasing sequence of the eigenvalue of the Dirichlet Laplacian $-\Delta_t$. Since $\kappa > 0$, the operator $(1 + |x|)^{-\kappa} \left(-\frac{d^2}{dx^2} + \lambda_\ell\right)^{-1/2}$ with $\ell \in {\mathbb N}$ fixed, is compact in ${\rm L}^2(\re)$ by \cite[Theorem 13, Section 8, Chapter 11]{BSo2}. On the other hand,
    $$
    \left\|(1 + |x|)^{-\kappa} \left(-\frac{d^2}{dx^2} + \lambda_\ell\right)^{-1/2}\right\| \leq \lambda_\ell^{-1/2}, \quad \ell \in {\mathbb N},
    $$
    and $\lim_{\ell \to \infty} \lambda_\ell^{-1/2} = 0$. Therefore, the orthogonal sum in \eqref{411} is compact.
    \end{proof}
    \begin{lemma} \label{l402}
    Let $\omega$, $\beta$, and $\eps$ satisfy the hypotheses of Lemma \ref{l401}. Assume that ${\mathcal E}_j^-$, $j \geq 1$ (resp., ${\mathcal E}_j^+$, $j \geq 0$) is a lower (resp., upper) regular edge point of a gap in $\sigma(H_\beta)$. Then the operators $ |g_{i,n,r}|^{\frac{1+\nu}{2}} H_\beta^n P_j^{\pm}$ are compact in ${\rm L}^2(\Omega)$.
    \end{lemma}
    \begin{proof}
    Since $ |g_{i,n,r}|^{\frac{1+\nu}{2}} H_\beta^n P_j^{\pm} =    |g_{i,n,r}|^{\frac{1+\nu}{2}} H_\beta^n \Phi^* {\mathcal P}_j^{\pm} \Phi$,
    it suffices to prove the compactness of
    \bel{412}
    |g_{i,n,r}|^{\frac{1+\nu}{2}} G_j^{\pm} (E_j^{\pm})^n : {\rm L}^2({\mathcal I}_j^{\pm}) \to {\rm L}^2(\Omega),
    \ee
    where $G_j^{\pm} : {\rm L}^2({\mathcal I}_j^{\pm}) \to {\rm L}^2(\Omega)$ is the operator with integral kernel
    $$
    \psi_j^{\pm}(x_t,x_3;k)e^{ix_3k}, \quad (x_t,x_3) \in \Omega, \quad k \in {\mathcal I}_j^{\pm}.
    $$
    Let us first prove that $G_j^{\pm}$ is bounded. To this end we will prove the boundedness of $(G_j^{\pm})^* : {\rm L}^2(\Omega) \to {\rm L}^2({\mathcal I}_j^{\pm})$; the argument is  similar to the proof of \cite[Lemma 3.1]{R}. We have
    $$
    ((G_j^{\pm})^*u)(k) = \int_\Omega  e^{-ix_3 k}\overline{\psi_j^{\pm}(x_t,x_3;k)} u(x_t, x_3) dx_t dx_3, \quad k  \in {\mathcal I}_j^{\pm}.
    $$
    Write $\psi_j^{\pm}(x_t,x_3;k)$ as a Fourier series with respect to $x_3$, i.e.
    $$
    \psi_j^{\pm}(x_t,x_3;k) = (2\pi)^{-1/2} \sum_{\ell \in {\mathbb Z}} \psi^{\pm}_{j,\ell}(x_t;k) e^{ix_3 \ell}
    $$
    with
    $$
    \sum_{\ell \in {\mathbb Z}} \int_\omega |\psi^{\pm}_{j,\ell}(x_t;k)|^2 dx_t = 1, \quad k \in {\mathcal I}_j^{\pm}.
    $$
    Then
    $$
    ((G_j^{\pm})^*u)(k) = \sum_{\ell \in {\mathbb Z}} \int_\omega \hat{u}(x_t,k+\ell) \overline{\psi^{\pm}_{j,\ell}(x_t)} dx_t,
    $$
    and, hence,
    $$
    \int_{{\mathcal I}_j^{\pm}} |((G_j^{\pm})^*u)(k)|^2 dk \leq
    \int_{{\mathcal I}_j^{\pm}} \left(\sum_{\ell \in {\mathbb Z}} \int_\omega |\hat{u}(x_t,k+\ell)|^2 dx_t\right)\left(\sum_{\ell \in {\mathbb Z}} \int_\omega |\psi^{\pm}_{j,\ell}(x_t;k)|^2 dx_t\right)dk \leq
    $$
    $$
    \sum_{\ell \in {\mathbb Z}} \int_{{\mathbb T}^*}  \int_\omega |\hat{u}(x_t,k+\ell)|^2 dx_t dk = \int_\re \int_\omega |\hat{u}(x_t,k)|^2 dx_t dk = \int_\Omega |u(\bx)|^2 d\bx
    $$
    which implies $\|(G_j^{\pm})^*\| =  \|G_j^{\pm}\| \leq 1$.
    Now fix $N \in {\mathbb N}$ and denote by $\chi_N$ the characteristic function of the interval $[-\pi N, \pi N]$. Write
    \bel{x26}
     |g_{i,n,r}|^{\frac{1+\nu}{2}} G_j^{\pm} (E_j^{\pm})^n = \chi_N(x_3) |g_{i,n,r}(x_3)|^{\frac{1+\nu}{2}} G_j^{\pm} (E_j^{\pm})^n + (1-\chi_N(x_3)) |g_{i,n,r}(x_3)|^{\frac{1+\nu}{2}} G_j^{\pm} (E_j^{\pm})^n.
    \ee
     We have
     $$
     \|\chi_N |g_{i,n,r}|^{\frac{1+\nu}{2}} G_j^{\pm} (E_j^{\pm})^n\|^2_{\rm HS} \leq C_1^2 N \int_\omega \int_{\mathbb T} \int_{{\mathcal I}_j^{\pm}} |\psi_j^{\pm}(x_t,x_3;k)|^2 dk dx_3 dx_t \leq C_1^2  N,
     $$
     where $\|\cdot\|_{\rm HS}$ denotes the Hilbert--Schmidt norm, and
     $$
     C_1 : = \sup_{x \in \re} |g_{i,n,r}(x)|^{\frac{1+\nu}{2}} \sup_{k \in {\mathcal I}_j^{\pm}} E_j^{\pm}(k)^n.
     $$
     Moreover,
$$
\|(1-\chi_N) |g_{i,n,r}|^{\frac{1+\nu}{2}} G_j^{\pm} (E_j^{\pm})^n\| \leq C_2 (1 + \pi N)^{-\frac{\alpha(1+\nu)}{2}} \|G_j^{\pm} \| \leq C_2 (1 + \pi N)^{-\frac{\alpha(1+\nu)}{2}}
$$
where
$$
C_2 : = \sup_{x \in \re} |(1+|x|)^{\alpha}g_{i,n,r}(x)|^{(1+\nu)/2} \sup_{k \in {\mathcal I}_j^{\pm}} E_j^{\pm}(k)^n.
 $$
 Thus, the operator $|g_{i,n,r}|^{\frac{1+\nu}{2}} G_j^{\pm} (E_j^{\pm})^n$ in \eqref{412} can be approximated in norm by compact operators, and hence it is compact itself.
\end{proof}
For $\nu \in (-1,1)$ set
    \bel{x50}
\rho_\nu(x_3) : = (1 + x_3^2)^{-\alpha(1+\nu)/2}, \quad x_3 \in \re.
    \ee
As usual, we will denote by the same symbol the multiplier by $\rho_\nu$, acting in ${\rm L}^2(\Omega)$ or in ${\rm L}^2(\re)$.
Now we are in position to prove the main result of this subsection.

\begin{pr} \label{p41} Under the hypotheses of Lemma \ref{l402}, there exists a $c_0 \geq 0$ independent of $\lambda$ such that for any $\nu \in (0,1)$ we have
    $$
    N_{(-\infty, -\lambda)}\left(P_j^-({\mathcal E}_j^- - H_{\beta-\eps} + c_0\sum_{n=0}^2 H_\beta^n \rho_\nu H_\beta^n)P_j^-\right) + O(1) \leq
    $$
    $$
    {\mathcal N}_j^-(\lambda) \leq
    $$
    \bel{426}
    N_{(-\infty, -\lambda)}\left(P_j^-({\mathcal E}_j^- - H_{\beta-\eps} - c_0\sum_{n=0}^2 H_\beta^n \rho_\nu H_\beta^n)P_j^-\right) + O(1),
    \ee
    or, respectively,
    $$
    N_{(-\infty, -\lambda)}\left(P_j^+(H_{\beta-\eps}-{\mathcal E}_j^+ + c_0\sum_{n=0}^2 H_\beta^n \rho_\nu H_\beta^n)P_j^+\right) + O(1) \leq
    $$
    $$
     {\mathcal N}_j^+(\lambda) \leq
    $$
    \bel{427}
    N_{(-\infty, -\lambda)}\left(P_j^+(H_{\beta-\eps}-{\mathcal E}_j^+ - c_0\sum_{n=0}^2 H_\beta^n \rho_\nu H_\beta^n)P_j^+\right) + O(1),
    \ee
    as $\lambda \downarrow 0$.
\end{pr}
\begin{proof}
Introduce the operators $A_{\pm} = A_{\pm;j} : {\rm L}^2(\Omega) \to {\rm L}^2(\Omega; {\mathbb C}^{19})$ and $B_{\pm} = B_{\pm;j} : {\rm L}^2(\Omega) \to {\rm L}^2(\Omega; {\mathbb C}^{19})$ by
$$
A_{\pm}u = \left\{|g_{i,n,r}|^{\frac{1+\nu}{2}}H_\beta^n P_j^{\pm}u\right\}_{i=1,\ldots,4; n=0,\ldots n_i; r=1,\ldots,r_{n,i}},
 $$
 $$
 B_{\pm}u = \left\{S_{i,n,r} Q_j^{\pm}u\right\}_{i=1,\ldots,4; n=0,\ldots n_i; r=1,\ldots,r_{n,i}},
$$
for $u \in {\rm L}^2(\Omega)$. By Lemmas \ref{l401} and \ref{l402}, the operators $A_{\pm}$ and $B_{\pm}$ are compact. \\
Let us now prove \eqref{426}. Taking into account \eqref{47} and \eqref{48}, we easily find that
$$
H_{\beta-\eps} = P^-_j(H_{\beta-\eps} + A_-^* A_-)P^-_j + Q^-_j(H_{\beta-\eps} + B_-^* B_-)Q^-_j -C_>^- =
$$
\bel{424}
P^-_j(H_{\beta-\eps} - A_-^* A_-)P^-_j + Q^-_j(H_{\beta-\eps} - B_-^* B_-)Q^-_j + C_<^-
\ee
where
$$
C_>^- = (A_-^* - B_-^*)(A_- - B_-), \quad C_<^- = (A_-^* + B_-^*)(A_- + B_-).
$$
Evidently, the operators $C_>^-$ and $C_<^-$ are compact and non-negative. Applying Lemma \ref{l43}, we get
$$
N_{({\mathcal E}_j^- + \lambda, {\mathcal E}-s)}(P^-_j(H_{\beta-\eps} - A_-^* A_-)P^-_j) + N_{({\mathcal E}_j^- + \lambda, {\mathcal E}-s)}(Q^-_j(H_{\beta-\eps} - B_-^* B_-)Q^-_j) - n_+(s;C_<^-) \leq
$$
$$
{\mathcal N}_j^-(\lambda) \leq
$$
\bel{425}
N_{({\mathcal E}_j^- + \lambda, {\mathcal E}+s)}(P^-_j(H_{\beta-\eps} + A_-^* A_-)P^-_j) + N_{({\mathcal E}_j^- + \lambda, {\mathcal E}+s)}(Q^-_j(H_{\beta-\eps} + B_-^* B_-)Q^-_j) + n_+(s;C_>^-),
\ee
where $s \in (0, \min\left\{{\mathcal E}_j^+ - {\mathcal E}, {\mathcal E} - {\mathcal E}_j^-\right\})$ and $\lambda \in (0, {\mathcal E} -{\mathcal E}_j^- - s)$, while the operators $P^-_j(H_{\beta-\eps} \pm  A_-^* A_-)P^-_j$ (resp., $Q^-_j(H_{\beta-\eps} \pm B_-^* B_-)Q^-_j$) are considered as operators with domain $P^-_j D(H_\beta)$ (resp., $Q^-_j D(H_\beta)$), self-adjoint in the Hilbert space $P^-_j {\rm L}^2(\Omega)$ (resp., $Q^-_j {\rm L}^2(\Omega)$). Further, by construction, $[{\mathcal E}_j^-, {\mathcal E}_j^+) \cap \sigma(Q^-_j H_\beta Q^-_j) = \emptyset$.
Due to the compactness of the operators $H_{\beta-\eps} - H_\beta  \pm B_-^* B_-$, we have
$$
[{\mathcal E}_j^-, {\mathcal E}_j^+) \cap \sigma_{\rm ess}(Q^-_j (H_{\beta-\eps} \pm  B_-^* B_-)Q^-_j) = \emptyset,
$$
and, hence,
\bel{420}
N_{({\mathcal E}_j^- + \lambda, {\mathcal E}\pm s)}(Q^-_j(H_{\beta-\eps} + B_-^* B_-)Q^-_j) = O(1), \quad \lambda \downarrow 0.
\ee
Next,
$\sigma_{\rm ess}({\mathcal E}_j^- - H_{\beta-\eps} \mp A_-^*A_-) \subset [0, \infty)$. Therefore,
$$
N_{({\mathcal E}_j^- + \lambda, {\mathcal E}\pm s)}(P^-_j(H_{\beta-\eps} \pm A_-^* A_-)P^-_j) =
N_{({\mathcal E}_j^- - {\mathcal E}\mp s, -\lambda)}(P^-_j({\mathcal E}_j^-  - H_{\beta-\eps} \mp A_-^* A_-)P^-_j) =
$$
$$
N_{(-\infty, -\lambda)}(P^-_j({\mathcal E}_j^-  - H_{\beta-\eps} \mp A_-^* A_-)P^-_j) -
N_{(-\infty, {\mathcal E}_j^- - {\mathcal E}\mp s]}(P^-_j({\mathcal E}_j^-  - H_{\beta-\eps} \mp A_-^* A_-)P^-_j) =
$$
\bel{421}
N_{(-\infty, -\lambda)}(P^-_j({\mathcal E}_j^-  - H_{\beta-\eps} \mp A_-^* A_-)P^-_j) + O(1), \quad \lambda \downarrow 0.
\ee
It is easy to check that there exists a constant $c_0 > 0$ such that
$$
P^-_j A_-^* A_- P^-_j \leq  c_0  \sum_{n=0}^2 P^-_j H_\beta^n \rho_\nu H_\beta^n P^-_j.
$$
Therefore,
$$
N_{(-\infty, -\lambda)}(P^-_j({\mathcal E}_j^-  - H_{\beta-\eps} - A_-^* A_-)P^-_j) \leq
$$
\bel{428a}
N_{(-\infty, -\lambda)}\left(P^-_j({\mathcal E}_j^-  - H_{\beta-\eps} - c_0 \sum_{n=0}^2 H_\beta^n \rho_\nu H_\beta^n)P^-_j\right),
\ee
$$
N_{(-\infty, -\lambda)}(P^-_j({\mathcal E}_j^-  - H_{\beta-\eps} + A_-^* A_-)P^-_j) \geq
$$
\bel{428}
N_{(-\infty, -\lambda)}\left(P^-_j({\mathcal E}_j^-  - H_{\beta-\eps} + c_0 \sum_{n=0}^2 H_\beta^n \rho_\nu H_\beta^n)P^-_j\right).
\ee
Finally, due to the compactness of the operators $C^-_<$ and $C^-_>$, we have
    \bel{423}
n_+(s;C^-_<) < \infty, \quad n_+(s;C^-_>) < \infty, \quad s>0.
    \ee
    Putting together \eqref{425} -- \eqref{423}, we obtain \eqref{426}. The proof of \eqref{427} is quite similar, so that we omit the details, and just point out that the analogue of \eqref{424} is
    $$
H_{\beta-\eps} = P^-_j(H_{\beta-\eps} - A_+^* A_+)P^-_j + Q^-_j(H_{\beta-\eps} + B_+^* B_+)Q^-_j + C_>^+ =
$$
$$
P^-_j(H_{\beta-\eps} + A_+^* A_+)P^-_j + Q^-_j(H_{\beta-\eps} + B_+^* B_+)Q^-_j - C_<^+
$$
where
$$
C_>^+ = (A_+^* + B_+^*)(A_+ + B_+), \quad C_<^+ = (A_+^* - B_+^*)(A_+ + B_+),
$$
    while the analogue of \eqref{425} is
    $$
N_{({\mathcal E}+s,{\mathcal E}_j^+ - \lambda)}(P^+_j(H_{\beta-\eps} + A_+^* A_+)P^+_j) + N_{({\mathcal E}+s,{\mathcal E}_j^+ - \lambda)}(Q^+_j(H_{\beta-\eps} + B_-^* B_-)Q^+_j) - n_+(s;C_<^+) \leq
$$
$$
{\mathcal N}_j^+(\lambda) \leq
$$
$$
N_{({\mathcal E}-s,{\mathcal E}_j^+ - \lambda)}(P^+_j(H_{\beta-\eps} - A_+^* A_+)P^+_j) + N_{({\mathcal E}+s,{\mathcal E}_j^+ - \lambda)}(Q^+_j(H_{\beta-\eps} - B_-^* B_-)Q^+_j) + n_+(s;C_>^+).
$$
\end{proof}

\subsection{Reduction to a Schr\"odinger-type operator}
\label{ss43}
Introduce the unitary operators $U_j^{\pm}: {\rm L}^2({\mathcal I}_j^{\pm}) \to P_j^{\pm}{\rm L}^2(\Omega)$ which act on $f \in {\rm L}^2({\mathcal I}_j^{\pm})$ as follows
$$
(U_j^{\pm})f)(x_t,x_3) = (\Phi^*\tilde{f}_j^{\pm})(x_t,x_3), \quad (x_t,x_3) \in \Omega,
$$
$$
\tilde{f}_j^{\pm}(x_t,x_3;k) = \left\{
\begin{array} {l}
\psi_j^{\pm}(x_t,x_3;k) f(x) \quad {\rm if} \quad (x_t,x_3) \in \omega \times {\mathbb T}, \quad k \in {\mathcal I}_j^{\pm},\\
0 \quad {\rm if} \quad (x_t,x_3) \in \omega \times {\mathbb T}, \quad k \in {\mathbb T}^* \setminus {\mathcal I}_j^{\pm}.
\end{array}
\right.
$$
Further, define $\Gamma^{\pm}_{j,\ell}: {\rm L}^2({\mathcal I}_j^{\pm}) \to {\rm L}^2(\Omega)$, $\ell = 0,\ldots,4$, as the operators with integral kernels
$$
e^{ix_3k}\gamma^{\pm}_{j,\ell}(x_t,x_3;k), \quad (x_t,x_3) \in \Omega, \quad k \in {\mathcal I}_j^{\pm},
$$
where
$$
\gamma_{j,0}^{\pm}(x_t,x_3;k) : = \partial_\varphi \psi_j^{\pm}(x_t,x_3;k),
$$
$$
\gamma_{j,1}^{\pm}(x_t,x_3;k) : = (\beta \partial_\varphi +\partial_3 + ik)\psi_j^{\pm}(x_t,x_3;k),
$$
$$
\gamma_{j,2+n}^{\pm}(x_t,x_3;k) : = \psi_j^{\pm}(x_t,x_3;k)(E_{\ell(j)}^\pm(k))^n, \quad n=0,1,2.
$$
Set
$$
T_{j,1}^{\pm}(c) : = \pm E_j^{\pm} \mp {\mathcal E}_j^{\pm} \mp  2{\rm Re} (\Gamma^{\pm}_{j,0})^* \eps \Gamma^{\pm}_{j,1}
\pm (\Gamma^{\pm}_{j,0})^* \eps^2 \Gamma^{\pm}_{j,0} - c \sum_{\ell=2}^4 (\Gamma^{\pm}_{j,\ell})^* \rho_\nu \Gamma^{\pm}_{j,\ell}, \quad c \in \re.
$$
{\em Remark}: If $\phi: \re \to {\mathbb C}$ is in a suitable class, then the operator $(\Gamma^\pm_{j,\ell})^* \phi \Gamma^\pm_{j,m} : L^2({\mathcal I}_j^{\pm}) \to L^2({\mathcal I}_j^{\pm})$ admits interpretation as a pseudodifferential operator ($\Psi$DO) with amplitude
$$
A(k,k'; x) : = 2\pi \phi(-x) \int_\omega \overline{\gamma_{j,\ell}^{\pm}(x_t,-x;k)} \gamma_{j,m}^{\pm}(x_t,-x;k') dx_t, \quad k,k' \in {\mathcal I}_j^{\pm},
\quad x \in \re,
$$
(see e.g. \cite[Eq. (23.8), Chapter IV]{shu}), i.e. as an integral operator with kernel
$$
\frac{1}{2\pi} \int_\re A(k,k'; x)e^{i(k-k')x} dx;
$$
note that here $k$ plays the role of the ``coordinate variable" while $x$ plays the role of the ``momentum variable". Even though we are in the simple situation where the underlying domain ${\mathcal I}_j^{\pm}$ is just a finite union of bounded intervals, some of the following arguments will be inspired by the general theory of $\Psi$DOs.\\

It is straightforward to check that

$$
P_j^{\pm}(\pm H_{\beta-\eps} \mp {\mathcal E}_j^{\pm}   - c \sum_{n=0}^2 H_\beta^n \rho_\nu H_\beta^n)P_j^{\pm} = U_j^{\pm} \, T_{j,1}^{\pm}(c)\,(U_j^{\pm})^*, \quad c \in \re.
$$
Therefore,
\bel{430}
N_{(-\infty, -\lambda)}(P_j^{\pm}(\pm H_{\beta-\eps} \mp {\mathcal E}_j^{\pm}- c \sum_{n=0}^2 H_\beta^n \rho_\nu H_\beta^n)P_j^{\pm}) = N_{(-\infty, -\lambda)}(T_{j,1}^{\pm}(c)), \quad \lambda >0, \quad c \in \re.
\ee
Further, introduce the multipliers
\bel{x80}
a_j^{\pm}(\lambda) : = \left(\pm E_{\ell(j)}^{\pm} \mp {\mathcal E}_j^{\pm}+\lambda\right)^{-1/2}, \quad \lambda > 0,
\ee
as well as the operators
 \bel{x20}
T^{\pm}_{j,2}(\lambda; c) : =  a_j^{\pm}(\lambda) \left(\pm  2{\rm Re} (\Gamma^{\pm}_{j,0})^* \eps \Gamma^{\pm}_{j,1}
\mp (\Gamma^{\pm}_{j,0})^* \eps^2 \Gamma^{\pm}_{j,0} + c \sum_{\ell =2}^4(\Gamma^{\pm}_{j,\ell})^* \rho_\nu \Gamma^{\pm}_{j,\ell}\right)a_j^{\pm}(\lambda),
 \ee
compact and self-adjoint in ${\rm L}^2({\mathcal I}_j^{\pm})$. Applying the Birman--Schwinger principle (see Lemma \ref{l41}), we get
    \bel{432}
    N_{(-\infty, -\lambda)}(T_{j,1}^{\pm}(c)) = n_+(1; T_{j,2}^{\pm}(\lambda; c)), \quad \lambda >0, \quad c \in \re.
    \ee
    Our next goal is to show that if we replace on the intervals ${\mathcal I}_{j,m}$, $m=1,\ldots,M_j^{\pm}$, the functions $\gamma_{j,i}(x_t,x_3,k)$ by their values at $k=k_{j,m}^\pm$, as well as  the functions  $E_{\ell(j)}^\pm(k) \mp {\mathcal E}_j^\pm$ (see \eqref{x80}) by their main asymptotic terms
    $\mu_{j,m}^\pm(k-k_{j,m}^\pm)^2$ as $k \to k_{j,m}^\pm$, we will make a negligible error in the asymptotic analysis of ${\mathcal N}_j^\pm(\lambda)$ as $\lambda \downarrow 0$. To this end,
     we define $\tilde{\Gamma}^{\pm}_{j,\ell}: {\rm L}^2({\mathcal I}_j^{\pm}) \to {\rm L}^2(\Omega)$, $\ell = 0,\ldots,4$, as the integral operators with integral kernels
$e^{ix_3k} \tilde{\gamma}^{\pm}_{j,\ell}(x_t,x_3;k)$, $(x_t,x_3) \in \Omega$, $k \in {\mathcal I}_j^{\pm}$,
where
$$
\tilde{\gamma}^{\pm}_{j,\ell}(x_t,x_3;k) = \sum_{m=1}^{M_j^{\pm}}\gamma^{\pm}_{j,\ell}(x_t,x_3;k_{j,m}^{\pm}) \chi^{\pm}_{j,m}(k),
$$
and $\chi^{\pm}_{j,m}$ is the characteristic function of the interval ${\mathcal I}_{j,m}^{\pm} = (k_{j,m}^{\pm}-\delta, k_{j,m}^{\pm}+\delta)$. Denote by
$\tilde{a}^{\pm}_j(\lambda)$, $\lambda>0$, the multiplier by
$$
\sum_{m=1}^{M_j^{\pm}}\left(\mu^{\pm}_{j,m}(k-k_{j,m}^{\pm})^2 + \lambda\right)^{-1/2} \chi^{\pm}_{j,m}(k), \quad k \in {\mathcal I}_j^{\pm},
$$
the quantities $\mu^{\pm}_{j,m}$ being introduced in \eqref{x25}. Define the operators
$$
\tilde{T}^{\pm}_{j,2}(\lambda; c) : =  \tilde{a}_j^{\pm}(\lambda) \left(\pm  2{\rm Re} (\tilde{\Gamma}^{\pm}_{j,0})^* \eps \tilde{\Gamma}^{\pm}_{j,1}
 + c \sum_{i=0}^4(\tilde{\Gamma}^{\pm}_{j,i})^* \rho_\nu \tilde{\Gamma}^{\pm}_{j,i}\right)\tilde{a}_j^{\pm}(\lambda), \quad \lambda>0, \quad c \in \re,
$$
compact and self-adjoint in ${\rm L}^2({\mathcal I}_j^{\pm})$.

 \begin{pr} \label{p42}
 Under the hypotheses of Lemma \ref{l402}, for any $c_0 \in \re$ there exists a constant $c_1 \geq 0$ independent of $\lambda$ such that for any $\nu \in (0,1)$, and $s\in (0,1)$, we have
 \bel{434}
 n_+(1+s; \tilde{T}_{j,2}^{\pm}(\lambda; -c_1)) + O_s(1) \leq n_+(1; T_{j,2}^{\pm}(\lambda; c_0)),
 \ee
 \bel{435}
  n_+(1; T_{j,2}^{\pm}(\lambda; c_0)) \leq n_+(1-s; \tilde{T}_{j,2}^{\pm}(\lambda; c_1)) + O_s(1) ,
 \ee
 as $\lambda \downarrow 0$.
 \end{pr}
 \begin{proof}
 For definiteness, let us prove \eqref{435}. It is easy to see that for any given $c_0 \in \re$ there exist constants $c_1, c_2 >0$ such that
 $$
 \pm  2{\rm Re} (\Gamma^{\pm}_{j,0})^* \eps \Gamma^{\pm}_{j,1}
\mp (\Gamma^{\pm}_{j,0})^* \eps^2 \Gamma^{\pm}_{j,0} + c_0 \sum_{\ell=2}^4(\Gamma^{\pm}_{j,\ell})^* \rho_\nu \Gamma^{\pm}_{j,\ell} \leq
$$
 \bel{436}
\pm 2{\rm Re} (\tilde{\Gamma}^{\pm}_{j,0})^* \eps \tilde{\Gamma}^{\pm}_{j,1}
 + c_1 \sum_{i=0}^4(\tilde{\Gamma}^{\pm}_{j,i})^* \rho_\nu \tilde{\Gamma}^{\pm}_{j,i} + c_2 \sum_{i=0}^4\left(\tilde{\Gamma}^{\pm}_{j,i}-\Gamma^{\pm}_{j,i}\right)^* \rho_{-\nu} \left(\tilde{\Gamma}^{\pm}_{j,i}-\Gamma^{\pm}_{j,i}\right) .
 \ee
 For a given $r \in (0,1)$, pick a $\delta > 0$, the semi-length of the intervals ${\mathcal I}_{j,m}^{\pm}$, so small that for each $\lambda > 0$ we have
 \bel{437}
 a_j^{\pm}(\lambda) \geq (1-r) \tilde{a}_j^{\pm}(\lambda)
 \ee
 on ${\mathcal I}_{j}^{\pm}$.
 Estimates \eqref{436}  -- \eqref{437}, the mini-max principle,  the Weyl inequalities \eqref{41}, identity \eqref{42}, and the Ky Fan inequalities \eqref{43} now imply
$$
n_+(1; T_{j,2}^{\pm}(\lambda; c_0)) \leq
$$
\bel{438}
n_+((1-r)^3; \tilde{T}_{j,2}^{\pm}(\lambda; c_1)) + \sum_{i=0}^4 n_*((1-r)(r/c_2)^{1/2}/5; \rho_{-\nu}^{1/2} \left(\tilde{\Gamma}^{\pm}_{j,i}-\Gamma^{\pm}_{j,i}\right)\tilde{a}_j^{\pm}(\lambda)).
\ee
Define ${\mathcal G}_{j,i}^{\pm} : {\rm L}^2({\mathcal I}_j^{\pm}) \to {\rm L}^2(\Omega)$, $i=0,\ldots,4$, as the  operator with kernel
$$
e^{ix_3k} \sum_{m=1}^{M_j^\pm} \frac{\gamma_{j,i}^\pm (\bx , k) - \gamma_{j,i}^\pm(\bx , k_{j,m}^{\pm})}{k-k_{j,m}^{\pm}} \chi_{j,m}^{\pm}(k), \quad k \in {\mathcal I}_j^{\pm}, \quad \bx = (x_t,x_3) \in \Omega. $$
Since
$$
|k-k_{j,m}^{\pm}| (\left(\mu^{\pm}_{j,m}(k-k_{j,m}^{\pm})^2 + \lambda\right)^{-1/2}\leq (\mu^{\pm}_{j,m})^{-1/2}, \quad k \in {\mathcal I}_{j,m}^{\pm}, \quad \lambda >0,
$$
we have
\bel{439a}
n_*(r; \rho_{-\nu}^{1/2} \left(\tilde{\Gamma}^{\pm}_{j,i}-\Gamma^{\pm}_{j,i}\right)\tilde{a}^{\pm}(\lambda)) \leq n_*(r(\mu^{\pm}_{j,m})^{1/2}; \rho_{-\nu}^{1/2} {\mathcal G}_{j,i}^{\pm}), \quad r>0, \quad \lambda>0.
\ee
Let us prove that the operators $\rho_{-\nu}^{1/2} {\mathcal G}_{j,i}^{\pm} :{\rm L}^2({\mathcal I}_j^{\pm}) \to {\rm L}^2(\Omega)$ are compact, arguing as in the proof of Lemma \ref{l402}. By analogy with \eqref{x26}, write
$$
\rho_{-\nu}^{1/2} {\mathcal G}_{j,i}^{\pm} = \chi_N\rho_{-\nu}^{1/2} {\mathcal G}_{j,i}^{\pm} + (1-\chi_N)\rho_{-\nu}^{1/2} {\mathcal G}_{j,i}^{\pm}.
$$
It is easy to check that
$$
\|\chi_N\rho_{-\nu}^{1/2} {\mathcal G}_{j,i}^{\pm}\|_{\rm HS}^2 \leq 2N\sup_{k \in {\mathcal I}_j^{\pm}}\int_{\omega}\int_{\mathbb T}|\partial_k \gamma^\pm_{j,i}(\bx, k)|^2 d\bx,
$$
$$
\|\chi_N\rho_{-\nu}^{1/2} {\mathcal G}_{j,i}^{\pm}\|^2 \leq (1+\pi N)^{-\alpha(1-\nu)}\sup_{k \in {\mathcal I}_j^{\pm}}\int_{\omega}\int_{\mathbb T}|\partial_k \gamma_{j,i}^\pm(\bx, k)|^2 d\bx,
$$
which implies the compactness of the operators $\rho_{-\nu} {\mathcal G}_{j,i}^{\pm}$; in particular,  we have
\bel{439}
n_*(s; \rho_{-\nu}^{1/2} {\mathcal G}_{j,i}^{\pm}) < \infty, \quad s>0.
\ee
Combining \eqref{438} -- \eqref{439}, we get \eqref{435}. The proof of \eqref{434} is  analogous.
\end{proof}
Next, define the unitary operator ${\mathcal W} : {\rm L}^2({\mathcal I}_j^{\pm}) \to {\rm L}^2((-\delta,\delta); {\mathbb C}^{M_j^{\pm}})$
by
$$
({\mathcal W}u)_m(k) : = u(k+k_{j,m}^{\pm}),  \quad k \in (-\delta, \delta), \quad m=1,\ldots,M_j^{\pm},
$$
for $u \in {\rm L}^2({\mathcal I}_j^{\pm})$. Set
$$
\eta_{j;m,n}^{\pm}(x_3) : = \int_\omega \overline{\gamma_{j,0}^{\pm}(x_t,x_3;k_{j,m}^{\pm})} \gamma_{j,1}^{\pm}(x_3, x_t;k_{j,n}^{\pm})dx_t =
$$
$$
\int_\omega \overline{\partial_\varphi \psi_{j}^{\pm}(x_t,x_3;k_{j,m}^{\pm})} (\beta(x_3)\partial_{\varphi} + \partial_3 + ik_{j,n}^{\pm})\psi_{j}^{\pm}(x_3, x_t;k_{j,n}^{\pm})dx_t,
$$
and
$$
\zeta_{j;m,n}^{\pm}(x_3) : = \sum_{i=0}^4 \int_\omega \overline{\gamma_{j,i}^{\pm}(x_t,x_3;k_{j,m}^{\pm})} \gamma_{j,i}^{\pm}(x_3, x_t;k_{j,n}^{\pm})dx_t, \quad x_3 \in {\mathbb T}, \quad m,n,=1,\ldots,M_j^{\pm};
$$
thus $2{\rm Re}\,\eta_{j;m,m}^{\pm}(x_3)$ coincides with function $\eta_{j;m}^{\pm}$ defined in \eqref{x1}.
Let $$T_{j,3}^{\pm}(\lambda,c) : {\rm L}^2((-\delta,\delta); {\mathbb C}^{M_j^{\pm}}) \to {\rm L}^2((-\delta,\delta); {\mathbb C}^{M_j^{\pm}})$$ be the operators with matrix-valued
integral kernels
\bel{442}
{\mathcal T}_j^{\pm}(k,k';\lambda,c) : = \left\{{\mathcal T}_{j;m,n}^{\pm}(k,k';\lambda,c)\right\}_{m,n=1}^{M_j^\pm}, \quad k,k' \in (-\delta,\delta),
\ee
with
$$
{\mathcal T}_{j;m,n}^{\pm}(k,k';\lambda,c) : =
$$
$$
\sqrt{2\pi} \tilde{a}_{j,m}(k;\lambda)\left(\mathcal{F}{(\pm\eps (\eta_{j;m,n}^{\pm} + \overline{\eta^{\pm}_{j;n,m}}) + c\rho_\nu \zeta_{j;m,n}^{\pm})}\right)(k-k'+k_{j,m}^{\pm}-k_{j,n}^{\pm})\tilde{a}_{j,n}(k';\lambda),
$$
where, as indicated in \eqref{x27}, $\mathcal{F}{(\pm\eps (\eta_{j;m,n}^{\pm} + \overline{\eta^{\pm}_{j;n,m}})+ c\rho_\nu \zeta_{j;m,n}^{\pm})}$ is the Fourier transform of the function $\pm\eps (\eta_{j;m,n}^{\pm} + \overline{\eta^{\pm}_{j;n,m}}) + c\rho_\nu \zeta_{j;m,n}^{\pm}$, and
$$
 \tilde{a}_{j,m}(k;\lambda) : = \left(\mu_{j,m}^{\pm}k^2 + \lambda\right)^{\-1/2}, \quad m=1,\ldots,m_j^{\pm}, \ldots k \in (-\delta, \delta).
 $$
 Then we have
 $$
 \tilde{T}^{\pm}_{j,2}(\lambda; c) = {\mathcal W}^* T^{\pm}_{j,3}(\lambda; c) {\mathcal W};
 $$
 in particular,
 \bel{x32}
 n_+(s;  \tilde{T}^{\pm}_{j,2}(\lambda; c)) =  n_+(s;T^{\pm}_{j,3}(\lambda; c)), \quad \lambda > 0, \quad c \in \re.
    \ee
 Our next goal is to show that if we omit the off-diagonal part of \eqref{442}, and replace in its diagonal part the functions $\eta_{j,m}^{\rm} : = 2{\rm Re} \eta_{j,m,m}^{\pm}$ and $\zeta_{j,m}^{\pm} : = \zeta_{j;m,m}$ by their mean values, we will make a negligible error in the asymptotic analysis of ${\mathcal N}_j^{\pm}(\lambda)$ as $\lambda \downarrow 0$.
 Let $t_{j,m,1}^{\pm}(\lambda,c) : {\rm L}^2(-\delta,\delta) \to {\rm L}^2(-\delta,\delta)$, $m=1,\ldots, M_j^{\pm}$, be the operators with integral kernels
 \bel{460}
 \tau_{j,m}^{\pm}(k,k´;\lambda,c) : = \sqrt{2\pi} \tilde{a}_{j,m}(k;\lambda)\left(\pm \langle \eta_{j,m}^{\pm}\rangle \hat{\eps}+ c \langle \zeta_{j,m}^{\pm}\rangle \hat{\rho_\nu}\right)(k-k') \tilde{a}_{j,m}(k';\lambda), \quad k,k' \in (-\delta,\delta).
 \ee
 \begin{pr} \label{p452}
 Under the hypotheses of Lemma \ref{l402}, for each $s > 0$, $r \in (0,1)$, and $c \in \re$, we have
 $$
\sum_{m=1}^{M_j^\pm} n_+(s(1+r); t_{j,m,1}^{\pm}(\lambda, c)) + O_{s,r}(1) \leq
$$
$$
n_+(s; T_{j,3}^{\pm}(\lambda, c)) \leq
$$
\bel{452}
 \sum_{m=1}^{M_j^\pm} n_+(s(1-r); t_{j,m.1}^{\pm}(\lambda, c))  + O_{s,r}(1), \quad  \lambda \downarrow 0.
\ee
\end{pr}
\begin{proof}
Set
$T_{j,4}^{\pm}(\lambda,c) : = \bigoplus_{m=1}^{M_j^{\pm}} t_{j,m,1}^{\pm}(\lambda, c)$.
Then,
\bel{454}
n_+(s; T_{j,4}^{\pm}(\lambda,c)) = \sum_{m=1}^{M_j^\pm} n_+(s; t_{j,m,1}^{\pm}(\lambda, c)), \quad s>0, \quad \lambda >0, \quad c \in \re.
\ee
On the other hand, the Weyl inequalities imply that for $s > 0$ and $r \in (0,1)$ we have
$$
n_+(s(1+r); T_{j,4}^{\pm}(\lambda,c)) - n_-(sr; T_{j,3}^{\pm}(\lambda,c) - T_{j,4}^{\pm}(\lambda,c)) \leq
$$
$$
n_+(s; T_{j,3}^{\pm}(\lambda,c)) \leq
$$
    \bel{455}
    n_+(s(1-r); T_{j,4}^{\pm}(\lambda,c)) + n_+(sr; T_{j,3}^{\pm}(\lambda,c) - T_{j,4}^{\pm}(\lambda,c)).
\ee
Bearing in mind \eqref{454} -- \eqref{455}, we find that in order to prove \eqref{452}, it suffices to show that for each $s>0$ we have
    \bel{x28}
     n_{+}(s; T_{j,3}^{\pm}(\lambda,c) - T_{j,4}^{\pm}(\lambda,c)) = O_s(1),
     \ee
     \bel{x28a}
     n_{-}(s; T_{j,3}^{\pm}(\lambda,c) - T_{j,4}^{\pm}(\lambda,c)) = O_s(1),
     \ee
     as $\lambda \downarrow 0$. Note that $T_{j,3}^{\pm}(\lambda,c) - T_{j,4}^{\pm}(\lambda,c) : {\rm L}^2((-\delta,\delta); {\mathbb C}^{M_j^{\pm}}) \to {\rm L}^2((-\delta,\delta); {\mathbb C}^{M_j^{\pm}})$
     can be written as an operator with matrix-valued integral kernel
     $$
\sqrt{2\pi} \tilde{a}_{j,m}(k;\lambda)\left(\delta_{m,n} \left(\mathcal{F}{(\pm\eps (\eta_{j;m}^{\pm} - \langle \eta_{j;m}^{\pm}\rangle)  + c\rho_\nu (\zeta_{j;m}^{\pm} - \langle \zeta_{j;m}^{\pm}\rangle))}\right)(k-k') + \right.
$$
 \bel{456}
\left. (1-\delta_{m,n}) \left(\mathcal{F}{(\pm\eps (\eta_{j;m,n}^{\pm} + \overline{\eta^{\pm}_{j;n,m}}) + c\rho_\nu \zeta_{j;m,n}^{\pm})}\right)(k-k'+k_{j,m}^{\pm}-k_{j,n}^{\pm})\right)\tilde{a}_{j,n}(k';\lambda),
\ee
with $k, k' \in (-\delta, \delta)$ and  $m=1,\ldots,M_j^{\pm}$. Pick
$$
\varkappa < \frac{1}{2}  \inf_{\ell \in {\mathbb Z}, \, m\neq n}|\ell + k_{j,m}^{\pm}-k_{j,n}^{\pm}| \in (0,1/2),
$$
 and $\delta \in
(0,\varkappa)$. Let $\Theta \in C_0^{\infty}(\re)$ be an even real-valued function such that ${\rm supp}\,\Theta \subset [-2\varkappa,2\varkappa]$, $\Theta(k) = 1$ for every $k \in [-2\delta, 2\delta]$ and
$\Theta(k) \in [0,1]$ for every $k \in \re$. Then we can multiply by $\Theta(k-k')$ the entries of the integral kernel of the operator $T_{j,3}^{\pm}(\lambda,c) - T_{j,4}^{\pm}(\lambda,c)$, defined in \eqref{456}, leaving them invariant. Therefore, the quadratic form of the operator $T_{j,3}^{\pm}(\lambda,c) - T_{j,4}^{\pm}(\lambda,c)$ can be considered as the restriction on ${\rm L}^2((-\delta,\delta); {\mathbb C}^{M_j^{\pm}})$ of the quadratic form of the operator
$$
{\mathcal F}({\mathcal D}^2+\lambda)^{-1/2} {\mathbb V}({\mathcal D}^2+\lambda)^{-1/2}{\mathcal F}^*,
 $$
 compact and self-adjoint in ${\rm L}^2(\re; {\mathbb C}^{M_j^{\pm}})$. Here
$$
{\mathcal D}^2 = {\mathcal D}^2_{j,\pm} : = - {\mathcal M}_j^{\pm} \frac{d^2}{dx^2},
$$
${\mathcal M}_j^{\pm}$ is the constant diagonal matrix $\left\{\mu_{j,m}^{\pm} \delta_{m,n}\right\}_{n,m=1}^{M_j^{\pm}}$, ${\mathbb V}$ is a matrix-valued potential with entries
$$
{\mathbb V}_{m,n}(x) : = 2\pi  \int_\re e^{ikx} \left(\delta_{m,n}\sum_{\ell \in {\mathbb Z}, \, \ell \neq 0}\left({\mathcal F}(\pm\eta^{\pm}_{j, m;\ell} \eps + c \zeta^{\pm}_{j, m;\ell}\rho_\nu)\right)(k-\ell) + \right.
$$
$$
\left. (1-\delta_{m,n}) \sum_{\ell \in {\mathbb Z}} \left({\mathcal F}(\pm(\eta^{\pm}_{j,m,n;\ell} + \overline{\eta^{\pm}_{j,n,m;-\ell}}) \eps + c \zeta^{\pm}_{j, m,n;\ell}\rho_\nu)\right)(k-\ell+k_{j,m}^{\pm}-k_{j,n}^{\pm})\right)\Theta(k) dk,
$$
$x\in \re$, $n,m =1,\ldots,M_j^{\pm}$, and $\eta^{\pm}_{j, m;\ell}, \zeta^{\pm}_{j, m;\ell}, \eta^{\pm}_{j,m,n;\ell}, \zeta^{\pm}_{j,m,n;\ell}$,  are the Fourier coefficients with respect of the system
$(2\pi)^{-1/2} e^{i\ell x}$, $x \in {\mathbb T}$, $\ell \in {\mathbb Z}$, respectively of the functions $\eta^{\pm}_{j, m}, \zeta^{\pm}_{j, m}, \eta^{\pm}_{j,m,n}$, and $\zeta^{\pm}_{j,m,n}$.
Bearing in mind the unitarity of ${\mathcal F}$, and applying the mini-max principle, and the Birman--Schwinger principle, we get
\bel{457}
 n_+(s; T_{j,3}^{\pm}(\lambda,c) - T_{j,4}^{\pm}(\lambda,c)) \leq N_{(-\infty,-\lambda)}({\mathcal D}^2 - s^{-1}{\mathbb V}), \quad s>0.
 \ee
 Since the series of the Fourier coefficients of the functions $\eta^{\pm}_{j, m}, \zeta^{\pm}_{j, m}, \eta^{\pm}_{j,n,m}$, and $\zeta^{\pm}_{j,n,m}$, are absolutely convergent, while Lemma \ref{l44} implies that the functions $\hat{\eps}(\cdot-\ell+k_{j,n}^{\pm}-k_{j,m}^{\pm})$, $\hat{\rho}_\nu(\cdot-\ell+k_{j,n}^{\pm}-k_{j,m}^{\pm})$, $\ell \in {\mathbb Z}$, $m \neq n$, and $\hat{\eps}(\cdot-\ell)$, $\hat{\rho}_\nu(\cdot-\ell)$, $\ell \in {\mathbb Z}$, $\ell \neq 0$, together with their derivatives of order up to three, are uniformly bounded on ${\rm supp}\,\Theta$, we have
 $$
 \|{\mathbb V}(x)\| = O\left((1+|x|)^{-3}\right), \quad x \in \re.
 $$
 Now Lemma \ref{l42} (iii) easily implies that
    \bel{458}
    N_{(-\infty,-\lambda)}({\mathcal D}^2 - s^{-1}{\mathbb V}) = O(1), \quad \lambda \downarrow 0, \quad s>0.
    \ee
    Putting together \eqref{457} and \eqref{458}, we obtain \eqref{x28}. The proof of \eqref{x28a} is analogous, and reduces to the replacement of ${\mathbb V}$ by $-{\mathbb V}$.
\end{proof}
Further, the quadratic forms of the operators $t_{j,m,1}^{\pm}(\lambda,c)$  $m=1,\ldots, M_j^{\pm}$, can be considered as restrictions on $ {\rm L}^2(-\delta,\delta)$ of the quadratic forms of
$t_{j,m,2}^{\pm}(\lambda,c\zeta_m)$,  where the operators
$$
t_{j,m,2}^{\pm}(\lambda,c) : =
 {\mathcal F}\left(-\mu_{j,m}^{\pm} \frac{d^2}{dx^2} + \lambda\right)^{-1/2}(\pm 2 \pi \langle \eta_{j,m}^{\pm}\rangle \eps  + c \rho_\nu)\left(-\mu_{j,m}^{\pm} \frac{d^2}{dx^2} + \lambda\right)^{-1/2}{\mathcal F}^*
$$
are compact and self-adjoint in ${\rm L}^2(\re)$. Applying the mini-max principle,  we get
\bel{459}
n_+(s;t_{j,m,1}^{\pm}(\lambda,c)) \leq  n_+(s;t_{j,m,2}^{\pm}(\lambda,c_1))
\ee
with $c_1 = 2\pi |\langle \zeta_{j,m}^{\pm}\rangle|$. Let us establish the corresponding lower bound. Define $t_{j,m,3}^{\pm}(\lambda,c) : {\rm L}^2(\re) \to {\rm L}^2(\re)$ as the operator with integral kernel
$$
\tau(k,k';\lambda,c) \chi_{(-\delta,\delta)}(k) \chi_{(-\delta,\delta)}(k'), \quad k,k' \in \re,
$$
(see \eqref{460}). Evidently, the non-zero eigenvalues of the operators $t_{j,m,1}^{\pm}(\lambda,c)$ and $t_{j,m,3}^{\pm}(\lambda,c)$ coincide, and we have
    \bel{x29}
    n_+(s;t_{j,m,1}^{\pm}(\lambda,c)) = n_+(s;t_{j,m,3}^{\pm}(\lambda,c)), \lambda > 0, \quad c \in \re.
    \ee
    On the other hand, it is easy to see that for each $c \in \re$ there exist constants $c_1, c_2 > 0$ such that
    \bel{x30}
    t_{j,m,3}^{\pm}(\lambda,c) \geq t_{j,m,2}^{\pm}(\lambda,-c_1) - c_2 (t_{j,m,4}^{\pm})^* t_{j,m,4}^{\pm},
    \ee
    where $t_{j,m,4}^{\pm} : {\rm L}^2(\re) \to {\rm L}^2(\re)$ is an operator with integral kernel
    $$
    (1 + x^2)^{-\alpha(1-\nu)/4} e^{ikx} \chi_{\re \setminus (-\delta,\delta)}(k) |k|^{-1/2}, \quad x \in \re, \quad k \in \re.
    $$
    Estimate \eqref{x30}, the mini-max principle and the Weyl inequalities imply
    \bel{462}
   n_+(s;t_{j,m,3}^{\pm}(\lambda,c)) \geq n_+(s(1+r);t_{j,m,2}^{\pm}(\lambda,-c_1)) - n_*(\sqrt{sr/c_2}; t_{j,m,4}^{\pm}).
   \ee
   By \cite[Theorem 13, Section 8, Chapter 11]{BSo2}, the operator $t_{j,m,4}^{\pm}$ is compact. Hence,
   \bel{464}
   n_*(s; t_{j,m,4}^{\pm}) < \infty, \quad s>0.
   \ee
   Now, the combination of \eqref{x29}, \eqref{462}, and \eqref{464} imply
   \bel{465}
   n_+(s;t_{j,m,1}^{\pm}(\lambda,c)) \geq n_+(s(1+r);t_{j,m,2}^{\pm}(\lambda,-c_1)) + O_{s,r}(1), \quad \lambda \downarrow 0.
   \ee
   Finally, the Birman-Schwinger principle implies
\bel{x31}
n_+(s;t_{j,m,2}^{\pm}(\lambda,c)) = N_{(-\infty,-\lambda)}\left(-\mu_{j,m}^{\pm} \frac{d^2}{dx^2} -  s^{-1}(\pm 2\pi \langle \eta_{j,m}^{\pm}\rangle \eps  + c \rho_\nu )\right).
\ee
   Putting together \eqref{426}, \eqref{427}, \eqref{430}, \eqref{432}, \eqref{434}, \eqref{435}, \eqref{x32}, \eqref{452}, \eqref{459}, \eqref{465}, and \eqref{x31}, we find that under the  hypotheses of Theorem \ref{th1}, there exists a constant $c>0$ such that for each $s \in (0,1)$ and $\nu \in (0,1)$, we have
   $$
   \sum_{m=1}^{M_j^\pm} N_{(-\infty,-\lambda)}\left(-\mu_{j,m}^{\pm} \frac{d^2}{dx^2} -  (1+s)^{-1}(\pm 2\pi \langle \eta_{j,m}^{\pm}\rangle \eps  - c \rho_\nu )\right) +O_s(1) \leq
   $$
   $$
   {\mathcal N}_j^{\pm}(\lambda) \leq
   $$
   \bel{467}
   \sum_{m=1}^{M_j^\pm} N_{(-\infty,-\lambda)}\left(-\mu_{j,m}^{\pm} \frac{d^2}{dx^2} -  (1-s)^{-1}(\pm 2\pi \langle \eta_{j,m}^{\pm}\rangle \eps  + c \rho_\nu )\right) + O_s(1),
   \ee
   as $\lambda \downarrow 0$. Now the results of Theorem \ref{th1} follow from \eqref{467} and Lemma \ref{l42}. For the convenience of the reader, we add just a few hints concerning the details:
   \begin{itemize}
   \item First of all, note that since $\nu > 0$ we have $\rho_\nu(x) = o(\eps(x))$ as $|x| \to \infty$.
   \item If $\pm \langle \eta_{j,m}^{\pm} \rangle > 0$ for some $m =1,\ldots,M_j^{\pm}$, then \eqref{21} follows from Lemma \ref{l42} (i). Here, we should also take into account the limiting relation
   $$
   \lim_{s \to 0} \lim_{\lambda \downarrow 0} \frac{\int_\re(\pm (1+s)^{-1} 2\pi \langle \eta_{j,m}^{\pm} \rangle \eps(x) - \lambda)_+^{-1/2} dx} {\int_\re(\pm 2\pi \langle \eta_{j,m}^{\pm} \rangle \eps(x) - \lambda)_+^{-1/2} dx} = 1.
   $$
   \item If $\pm \langle \eta_{j,m}^{\pm} \rangle < 0$ for all $m =1,\ldots,M_j^{\pm}$, and $\eps \in {\mathcal S}_{4,\alpha}^+(\re)$, then the positive part of the function $\pm 2\pi \langle \eta_{j,m}^{\pm}\rangle \eps  + c \rho_\nu$  in \eqref{467} has a compact support since $\nu > 0$. Therefore, in this case \eqref{22} follows from Lemma \ref{l42} (iii).
   \item If $\langle \eta_{j,m}^{\pm} \rangle = 0$ for some $m =1,\ldots,M_j^{\pm}$, then the only non-zero term of the potential in \eqref{467} is proportional to $\rho_\nu$. If $\alpha > 1$ then we can pick $\nu \in (0,1)$ so that $\alpha(1+\nu) > 2$, and in this case \eqref{22} follows again from Lemma \ref{l42} (iii). If $\alpha \in (0,1]$, then \eqref{x2} follows from Lemma \ref{l42} (i) and the fact that $(1+\nu)\alpha$ could be chosen arbitrarily close, but yet smaller than $2\alpha$.
       \item If $\alpha = 2$, Theorem \ref{th1} (iii) follows  from Lemma \ref{l42} (ii).
   \item Finally, if $\alpha > 2$ (and, hence, $\alpha(1+\nu) > 2$), then Theorem \ref{th1} (iv) follows immediately from Lemma \ref{l42} (iii).

   \end{itemize}

\appendix
\section{Proof of Proposition \ref{p21}}
\label{app} \setcounter{equation}{0}
Assume the hypotheses of Proposition \ref{p21} (i) - (iii). By the Birman--Schwinger principle
    \bel{x37}
    N_{(-\infty,-\lambda)}(h_{\rm eff}) = n_+(1; a(\lambda) {\mathcal F} \eta \eps {\mathcal F}^* a(\lambda)), \quad \lambda > 0,
    \ee
    where $a(k;\lambda) : = (\mu k^2 + \lambda)^{-1/2}$, $k \in \re$, $\lambda \geq 0$. Denote by $\chi_1$ the characteristic function of the interval $(-\delta,\delta)$ with $\delta \in (0,1/2)$. Set $\chi_2 : = 1 - \chi_1$ and write
    $$
    a(\lambda) {\mathcal F} \eta \eps {\mathcal F}^* a(\lambda) =
    $$
    \bel{x38}
   \langle \eta \rangle a(\lambda) {\mathcal F}  \eps {\mathcal F}^* a(\lambda) + \sum_{j=1,2}\,a(\lambda)\chi_j {\mathcal F}( \eta -  \langle \eta \rangle)\eps {\mathcal F}^* \chi_j a(\lambda) + 2{\rm Re}\,a(\lambda)\chi_1 {\mathcal F}( \eta -  \langle \eta \rangle)\eps {\mathcal F}^* \chi_2 a(\lambda).
   \ee
   Further, for any $u \in {\rm L}^2(\re)$,
   $$
   \left( \left(2{\rm Re}\,a(\lambda)\chi_1 {\mathcal F}( \eta -  \langle \eta \rangle)\eps {\mathcal F}^* \chi_2 a(\lambda)\right)u, u\right)_{{\rm L}^2(\re)} = 2 {\rm Re} \left( f,g\right)_{{\rm L}^2(\re)}
   $$
   where $\left(\cdot, \cdot\right)_{{\rm L}^2(\re)}$ is the scalar product in ${\rm L}^2(\re)$, and
   $$
   f :  = \rho_\nu^{1/2} {\mathcal F}^* \chi_1 a(\lambda)u, \quad g: = \rho_{-\nu}^{1/2} ( \eta -  \langle \eta \rangle)\eps \rho_0^{-1}{\mathcal F}^* \chi_2 a(\lambda)u, \quad \nu \in (0,1),
   $$
   the multiplier $\rho_\nu$, $\nu \in (-1,1)$, being defined in \eqref{x50}.
   Evidently, since $\rho_\nu(x) \leq \rho_{-\nu}(x)$, $x \in \re$, $\nu \in (0,1)$, we have
   $$
   -\|\rho_\nu^{1/2} {\mathcal F}^*  a(\lambda)u\|^2 - (1+2C^2) \|\rho_{-\nu}^{1/2} {\mathcal F}^* \chi_2 a(\lambda)u\|^2 \leq -\frac{1}{2}\|f\|^2 - 2 \|g\|^2 \leq
   $$
   $$
   2 {\rm Re} \left( f,g\right)_{{\rm L}^2(\re)} \leq
   $$
   $$
   \frac{1}{2}\|f\|^2 + 2\|g\|^2 \leq \|\rho_\nu^{1/2} {\mathcal F}^*  a(\lambda)u\|^2 + (1+2C^2) \|\rho_{-\nu}^{1/2} {\mathcal F}^* \chi_2 a(\lambda)u\|^2
   $$
   with
    $$
    C : = \sup_{x \in \re}|\eta(x) -  \langle \eta \rangle| \rho_0(x)^{-1}| \eps(x)|.
    $$
     Therefore,
    $$
   - a(\lambda) {\mathcal F} \rho_\nu {\mathcal F}^*  a(\lambda) - (1+2C^2) a(\lambda)\chi_2 {\mathcal F} \rho_{-\nu} {\mathcal F}^* \chi_2 a(\lambda) \leq $$
   $$
2{\rm Re}\,a(\lambda)\chi_1 {\mathcal F}( \eta -  \langle \eta \rangle)\eps {\mathcal F}^* \chi_2 a(\lambda) \leq
$$
    \bel{x39}
    a(\lambda) {\mathcal F} \rho_\nu {\mathcal F}^*  a(\lambda) + (1+2C^2) a(\lambda)\chi_2 {\mathcal F} \rho_{-\nu} {\mathcal F}^* \chi_2 a(\lambda),  \quad \nu \in (0,1).
    \ee
    Similarly,
    $$
   -C a(\lambda)\chi_2 {\mathcal F}\rho_{-\nu} {\mathcal F}^*\chi_2  a(\lambda) \leq  -C a(\lambda)\chi_2 {\mathcal F}\rho_0 {\mathcal F}^*\chi_2  a(\lambda) \leq $$
    $$
    a(\lambda)\chi_2 {\mathcal F}( \eta -  \langle \eta \rangle) \eps {\mathcal F}^*\chi_2  a(\lambda) \leq
    $$
    \bel{dec1}
    C a(\lambda)\chi_2 {\mathcal F}\rho_0 {\mathcal F}^*\chi_2  a(\lambda) \leq C a(\lambda)\chi_2 {\mathcal F}\rho_{-\nu} {\mathcal F}^*\chi_2  a(\lambda), \quad \nu \in (0,1).
    \ee
    Now it follows from \eqref{x38} -- \eqref{dec1} that
$$
a(\lambda) {\mathcal F}(\langle \eta \rangle \eps - \rho_\nu) {\mathcal F}^*  a(\lambda) + a(\lambda)\chi_1 {\mathcal F}(\eta - \langle \eta \rangle)  \eps {\mathcal F}^* \chi_1 a(\lambda) -
$$
$$
(1 + C + 2C^2)a(\lambda)\chi_2 {\mathcal F} \rho_{-\nu} {\mathcal F}^* \chi_2 a(\lambda) \leq
$$
$$
a(\lambda) {\mathcal F} \eta \eps {\mathcal F}^* a(\lambda) \leq
$$
    $$
    a(\lambda) {\mathcal F}(\langle \eta \rangle \eps + \rho_\nu) {\mathcal F}^*  a(\lambda) + a(\lambda) {\mathcal F}(\eta - \langle \eta \rangle)  \eps {\mathcal F}^* \chi_1 a(\lambda) +
    $$
    \bel{x40}
    (1+C +2 C^2)a(\lambda)\chi_2 {\mathcal F} \rho_{-\nu} {\mathcal F}^* \chi_2 a(\lambda).
    \ee
    Applying the mini-max principle and the Weyl inequalities,  
    we find that \eqref{x40} implies
     $$
     n_+(1+s; a(\lambda) {\mathcal F}(\langle \eta \rangle \eps - \rho_\nu) {\mathcal F}^*  a(\lambda))
     - n_-(s/2; a(\lambda)\chi_1 {\mathcal F}(\eta - \langle \eta \rangle)  \eps {\mathcal F}^* \chi_1 a(\lambda)) -
     $$
     $$
     n_*(\sqrt{s/(2(1 + C+2C^2))}; \rho_{-\nu}^{1/2} {\mathcal F}^* \chi_2 a(0)) \leq
     $$
     $$
     n_+(1; a(\lambda) {\mathcal F} \eta \eps {\mathcal F}^* a(\lambda)) \leq
     $$
     $$
     n_+(1-s; a(\lambda) {\mathcal F}(\langle \eta \rangle \eps + \rho_\nu) {\mathcal F}^*  a(\lambda))
     + n_+(s/2; a(\lambda)\chi_1 {\mathcal F}(\eta - \langle \eta \rangle)  \eps {\mathcal F}^* \chi_1 a(\lambda)) +
     $$
     \bel{x41}
     n_*(\sqrt{s/(2(1+C+2C^2))}; \rho_{-\nu}^{1/2} {\mathcal F}^* \chi_2 a(0)), \quad s \in (0,1), 
     \ee 
     bearing in mind that
    $a(k;\lambda) \leq a(k;0)$ for $k \in {\rm supp}\,\chi_2$ and $\lambda \geq 0$. \\
     The operator $a(\lambda)\chi_1 {\mathcal F}(\eta - \langle \eta \rangle)  \eps {\mathcal F}^* \chi_1 a(\lambda)$ admits the integral kernel
     \bel{x42}
     (2\pi)^{-1} a(k;\lambda)\chi_1(k) \sum_{\ell \in {\mathbb Z}\setminus\{0\}}\eta_\ell  \hat{\eps}(k-k'-\ell) \chi_1(k') a(k';\lambda), \quad k,k' \in \re.
     \ee
     Let $\varkappa \in (\delta,1/2)$, and let $\Theta \in C_0^\infty(\re)$ with ${\rm supp}\,\Theta = [-2\varkappa, 2\varkappa]$ and
     ${\rm supp}\,(1-\Theta) \subset \re \setminus (-2\delta, 2\delta)$, be the real even function used in the proof of Proposition \ref{p452}. We can multiply the integral kernel in \eqref{x42} by $\Theta(k-k')$ without modifying it. Hence, by the mini-max principle and the Birman--Schwinger principle, we have
     $$
     n_\pm(s; a(\lambda)\chi_1 {\mathcal F}(\eta - \langle \eta \rangle)  \eps {\mathcal F}^* \chi_1 a(\lambda)) \leq
     $$
     $$
     n_\pm(s; (2\pi)^{-1/2} a(\lambda) {\mathcal F}(((\eta - \langle \eta \rangle)\eps)*\hat{\Theta}) {\mathcal F}^*  a(\lambda)) =
     $$
     \bel{x43}
     N_{(-\infty,-\lambda)}\left(-\mu \frac{d^2}{dx^2} \mp s^{-1} (2\pi)^{-1/2}(((\eta - \langle \eta \rangle)\eps)*\hat{\Theta})\right), \quad s>0, \quad \lambda >0.
     \ee
     Arguing as in the proof of Proposition \ref{p452}, we find with the help of Lemma \ref{l44} that
     \bel{x44}
     \left|(((\eta - \langle \eta \rangle)\eps)*\hat{\Theta})(x)\right| = O((1+|x|)^{-3}), \quad x \in \re.
     \ee
     Estimates \eqref{x43} -- \eqref{x44} combined with Lemma \ref{l42} (iii), imply
     \bel{x45}
     n_\pm(s; a(\lambda)\chi_1 {\mathcal F}(\eta - \langle \eta \rangle)  \eps {\mathcal F}^* \chi_1 a(\lambda)) = O_s(1), \quad \lambda \downarrow 0, \quad s>0. \ee
     Finally, the operator $\rho_{-\nu}^{1/2} {\mathcal F}^* \chi_2 a(0)$ with $\nu < 1$ is compact by \cite[Theorem 13, Section 8, Chapter 11]{BSo2}. Therefore, \bel{x46}
     n_*(s;\rho_{-\nu}^{1/2} {\mathcal F}^* \chi_2 a(0)) < \infty, \quad s>0.
     \ee
     Putting together \eqref{x41}, \eqref{x45}, and \eqref{x46}, and applying the Birman-Schwinger principle, we obtain
     $$
     N_{(-\infty,-\lambda)}\left(-\mu \frac{d^2}{dx^2} - (1+s)^{-1} (\langle \eta \rangle\eps - \rho_\nu)\right) + O_s(1) \leq
     $$
     $$
     n_+(1; a(\lambda) {\mathcal F} \eta \eps {\mathcal F}^* a(\lambda)) \leq
     $$
     \bel{x47}
      N_{(-\infty,-\lambda)}\left(-\mu \frac{d^2}{dx^2} - (1-s)^{-1} (\langle \eta \rangle\eps + \rho_\nu)\right) + O_s(1), \quad \lambda \downarrow 0,
      \ee
      for any $s \in (0,1)$ and $\nu \in (0,1)$. Now parts (i) - (iii) of Proposition \ref{p21} follow from estimates \eqref{x37} and \eqref{x47}, and Lemma \ref{l42}. Part (iv) of this proposition is implied directly by Lemma \ref{l42} (iii). \\

{\bf Acknowledgements}. The author thanks Hynek  Kova\v{r}\'{\i}k for several valuable remarks, and the anonymous referees for  the suggestions which contributed to the improvement of the exposition. The partial
support of the Chilean Scientific Foundation {\em Fondecyt}
under Grant 1130591, and of {\em N\'ucleo Milenio de F\'isica Matem\'atica} RC120002, is gratefully acknowledged.


{\sc G. Raikov}\\
Facultad de Matem\'aticas\\
Pontificia Universidad Cat\'olica de Chile\\
Av. Vicu\~na Mackenna 4860\\ Santiago de Chile\\
E-mail: graikov@mat.puc.cl\\

\end{document}